\documentclass[11pt]{article}
\usepackage{latexsym,amsfonts,amssymb,amsmath,euscript,amsthm}
\usepackage[applemac]{inputenc}  
\usepackage[english,francais]{babel}  
\usepackage{graphicx} 
\usepackage{subfigure}
\usepackage{hyperref}         
\usepackage{color,fancybox}
\usepackage{float}

\newcommand{\R} {\mathbb{R}}
\newcommand{\N} {\mathbb{N}}

\newcommand{\eps}{\epsilon}

\title{Unfolding operator method for thin domains with a locally periodic highly oscillatory boundary\footnote{
This research has been partially supported by grant MTM2012-31298, MINECO, Spain and Grupo de Investigaci\'on CADEDIF, UCM. The second author, Manuel Villanueva-Pesqueira, also partially supported by a FPU fellowship (AP2010-0786) from the Goverment of Spain.}   }
\author{Jos\'{e} M. Arrieta\footnote{Departamento de Matem\'atica Aplicada, Universidad Complutense de Madrid, 28040 Madrid and Instituto de Ciencias Matem\'aticas
CSIC-UAM-UC3M-UCM, C/Nicol\'as Cabrera 13-15, Cantoblanco, 28049 Madrid, Spain}
and Manuel Villanueva-Pesqueira\footnote{Departamento de Matem\'atica Aplicada, Universidad Complutense de Madrid, 28040 Madrid, Spain} 
}

\date{ }

\begin{document}

\maketitle
%
%
%
%

{\footnotesize 
\par\noindent {\bf Abstract:}
We analyze the behavior of solutions of the Poisson equation with homogeneous Neumann boundary conditions in a two-dimensional
thin domain which presents locally periodic oscillations at the boundary.  The oscillations are such that both the amplitude and period of the oscillations may vary in space.  We obtain the homogenized limit problem 
and a corrector result by extending the unfolding operator method to the case of locally periodic media. We emphasize the fact that the techniques developed in this paper can be adapted to other locally periodic  cases like reticulated or perforated domains where the period may be space-dependent. \vskip 0.5\baselineskip

%

\vspace{11pt}


\vspace{6pt}

}

\numberwithin{equation}{section}
\newtheorem{theorem}{Theorem}[section]
\newtheorem{lemma}[theorem]{Lemma}
\newtheorem{corollary}[theorem]{Corollary}
\newtheorem{proposition}[theorem]{Proposition}
\newtheorem{definition}[theorem]{Definition}
\newtheorem{remark}[theorem]{Remark}
\allowdisplaybreaks

\section{Introduction}
\selectlanguage{english}
In this paper we study the behavior of the solutions of  the Neumann problem for the Laplace operator
\begin{equation} \label{OP0}
\left\{
\begin{gathered}
- \Delta u^\epsilon + u^\epsilon = f^\epsilon
\quad \textrm{ in } R^\epsilon \\
\frac{\partial u^\epsilon}{\partial \nu ^\epsilon} = 0
\quad \textrm{ on } \partial R^\epsilon
\end{gathered}
\right. 
\end{equation}
where $f^\epsilon \in L^2(R^\epsilon)$,  $\nu^\epsilon = (\nu^\epsilon_1, \nu^\epsilon_2)$ is the unit outward normal to $\partial R^\epsilon$
and $\frac{\partial }{\partial \nu^\epsilon}$ is the outside normal derivative. 
The domain $R^\eps$ is a two dimensional thin domain which presents a highly oscillatory behavior at the top boundary,
 given by
 \begin{equation}\label{Def-thindomain}
  R^\epsilon = \{ (x, y) \in \R^2 \; | \;  x \in (0,1), \, 0 < y < \epsilon \, G(x, x/\eps) \}
  \end{equation}
where the smooth function $G$, defined as 
\begin{equation*}
\begin{array}{rl}
G:(0,1)\times \R &\longrightarrow (0,+\infty) \\
 (x,y)&\longrightarrow G(x,y)
 \end{array}
 \end{equation*}
 satisfies that  there exist two positive constants $G_0$, $G_1$ with 
\begin{equation*}
\begin{gathered}
0< G_0 \le G(x,y) \le G_1, \quad  \forall (x,y)\in (0,1)\times \R.
\end{gathered}
\end{equation*}  
Moreover,  for each $x\in (0,1)$, the function   $G(x, \cdot)$ is $l(x)$-periodic, with the function $l(\cdot)$ not being necessarily constant. This includes the case where the thin domain is locally periodic with constant period, for instance,
$G(x, x/\eps) = a(x) + b(x)g(x/\eps)$ where $a,b: (0,1) \to \R$ are $\mathcal{C}^1$ functions and $g: \R \to \R$ is an L-periodic smooth function (see for instance \cite{ArrPer2011}).  But, it also includes the very important and relevant case where the period changes as we vary $x$. For instance, $G(x, x/\eps) = a(x) + b(x)g(l(x)x/\eps)$ where $l: \R \to \R$ is certain smooth function.  

\begin{figure}
  \centering{\includegraphics [width=8cm, height=2.5cm]{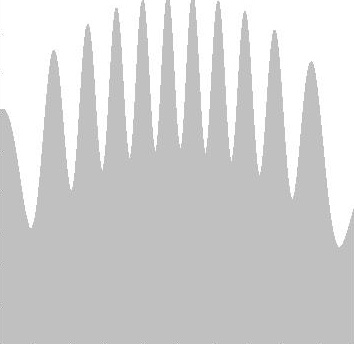}
    \caption{Model thin domain $R^\eps$}}
    \label{thin1}
\end{figure}

 The main novelty in this work is that we consider domains where both the amplitude and frequency of the oscillations depend on $x$ (see Figure 1). In this respect we are deviating from the purely periodic case, which is the most common setting in homogenization theory and we are interested in analyzing how the geometry of the domain,
 the varying amplitude and period of the function $G$, affects the homogenized limit problem. 

\par

The purely periodic case can be addressed by somehow standard techniques in homogenization theory, as developed 
in \cite{BenLioPap, CioPau, SP, CiorDonato}. If $G_\eps(x)=G(x/\eps)$ where $G$ is an $L$-periodic $C^1$ function and if we denote by
$$
Y^* = \{ (y_1,y_2) \in \R^2 \; : \; 0< y_1 < L, \quad 0 < y_2 < G(y_1) \}
$$
then the limit equation is shown to be 
\begin{equation} \label{GLP}
\left\{
\begin{gathered}
-q_0w_{xx} + w =  f(x), \quad x \in (0,1)\\
w'(0)=w'(1)=0
\end{gathered}
\right.
\end{equation} 
where 
$$
\begin{gathered}
q_0=\frac{1}{|Y^*|} \int_{Y^*} \Big\{ 1 - \frac{\partial X}{\partial y_1}(y_1,y_2) \Big\} dy_1 dy_2, \qquad 
\end{gathered}
$$ 
and $X$ is the unique solution of certain PDE problem posed in the basic cell $Y^*$. 
We refer to \cite{MP2, ArrPer2011,ArrCarPerSil} for details.  

Also, the case where the function $G(x, \cdot)$ is $L$-periodic with $L$ independent of $x$, that is, locally periodic case with fixed period,  was analyzed in \cite{ArrPer2011}. With the method of oscillatory test functions applied first to the case of piecewise periodic case and then with a domain perturbation argument, the following limit problem was obtained:

\begin{equation} \label{GLP1}
\left\{
\begin{gathered}
-\frac{1}{p(x)}(q(x)w_{x})_x + w =  f(x), \quad x \in (0,1)\\
w'(0)=w'(1)=0
\end{gathered}
\right.
\end{equation} 
where 
\begin{equation}\label{GLP2}
\begin{gathered}
q(x)= \int_{Y^*(x)} \Big\{ 1 - \frac{\partial X(x)}{\partial y_1}(y_1,y_2) \Big\} dy_1 dy_2, \qquad 
\end{gathered}
\end{equation}
\begin{equation}\label{GLP3}
p(x)=|Y^*(x)|
\end{equation}
and $X(x)(\cdot,\cdot)$ is the unique solution of an appropriate PDE problem posed in the basic cell $Y^*(x)$ which depends on the variable $x$ and 
it is given by 
$$
Y^*(x)= \{ (y_1,y_2) \in \R^2 \; : \; 0< y_1 < L, \quad 0 < y_2 < G(x,y_1) \}.
$$
Notice that if we assume that $Y^*(x)\equiv Y^*$, then we recover the homogenized problem in the purely periodic case. 
 
The analysis performed in \cite{ArrPer2011} used in a very definite way that the period of the function $G(x,\cdot)$ is independent of $x$ and the techniques do not apply in a very straightforward way to the more general case of a varying period. 


%
 
In the literature there are several, although not many, works on homegenization in locally periodic structures. In \cite{CP, MN, NM} an asymptotic expansion technique was used  to obtain the limit problem and the estimates of the rate of the convergence for problems defined in domains with locally periodic perforations,  i.e. the geometry of the cavities varies with space. Two scale convergence was applied in \cite{CMT,MPO}  to homogenize the warping, the torsion and the Neumann problems in two dimensional domains with  smoothy changing holes and in \cite{MA} two scale convergence was generalized to a locally periodic and fibrous media. All these works deal with the case where the basic cell varies with space but the period in which the basic cell is spaced is constant and do not vary with space. 

Let us also point out that there are several papers addressing the problem of studying the effect of rough boundaries on the behavior of the solution of partial differential equations.  Among others, we can mention \cite{Ac, Br, C1, Ge, J, komo} in the context of fluid flows and \cite{Ch1, Ch2} where complete asymptotic expansions of the solutions were studied.

When dealing with a thin domain $R^\eps$ as defined in \eqref{Def-thindomain},  where the function $G(x,\cdot)$ is periodic of period $l(x)$, we may distinguish two different situations.  

On one hand, if the function $h(x)=\frac{x}{l(x)}$ satisfies  $h'(x)>0$ for all $x\in [0,1]$ then 
$h:(0,1)\to \Big(0, \frac{1}{l(1)}\Big)$ is a diffeomorphism. In this particular case, if seems reasonable to perform the following change of variables 
\begin{equation*}
\begin{array}{rl}
Z^\eps: R^\eps &\longrightarrow R_1^\eps \\
 (x,y)&\longrightarrow (x_1, x_2) := \big(h(x), y\big) =\Big(\frac{x}{l(x)}, y\Big).
 \end{array}
 \end{equation*}
which transforms the thin domain $R^\eps$ into another thin domain, $R_1^\eps$, having an oscillatory boundary given by the $1-$periodic function $H(x,y)= G(h^{-1}(x), l(h^{-1}(x)) y)$, that is 
$$R_1^\epsilon = \Big\{ (x_1,x_2) \in \R^2 \; | \;  x_1 \in (0,1/l(1)),  \; 0 < x_2 < \epsilon \, H(x_1,x_1/\eps) \Big\}.$$ 

In this new system of coordinates, problem \eqref{OP0} is transformed into 
\begin{equation}\label{transformed-problem}
\left\{
\begin{gathered}
- h'(h^{-1}(x_1))\hbox{div}\big(B(v^\eps)\big)  + v^\epsilon = f^\epsilon
\quad \textrm{ in } R_1^\epsilon, \\
B(v^\eps)\cdot \eta  = 0
\quad \textrm{ on } \partial R_1^\epsilon,\\
v^\eps=u^\eps \circ (Z^\eps)^{-1} \textrm{ in } R_1^\epsilon
\end{gathered}
\right. 
\end{equation}
where $\eta$ denotes the unit outward normal vector field to $\partial R_1^\epsilon$ and 
\begin{equation*}
B(v) = \Big( h'(h^{-1}(x_1))\frac{ \partial v}{\partial x_1}, \frac{1}{ h'(h^{-1}(x_1))}\frac{ \partial v}{\partial x_2}\Big).
\end{equation*}
Under this change of variables, we have transformed the thin domain with variable amplitude and period of the oscillations into a locally periodic thin domain with constant period, although the amplitude continues to vary with space (that is, the function $H(x,y)$ actually depends on $x$). Moreover, in this case, the transformed differential operator, see \eqref{transformed-problem},  is now with non-constant coefficients.  Now, we may proceed using the techniques from \cite{ArrPer2011} or try to use the unfolding operator method adapted to the situation of variable amplitude, which is a particular case of what we are developing in this paper. 
\par\medskip

On the other hand, if the function $h(x)=\frac{x}{l(x)}$  is not a diffeomorphism, we cannot perform this changes of variables.  This is the situation that we are considering in this paper. As a matter of fact we will assume the following hypothesis 
\par\medskip
\begin{itemize}
\item[\bf{(H)}] The function $l(\cdot)$ is $\mathcal{C}^1$ and there exist  positive constants $l_0, l_1$ such that  $0<l_0 \leqslant l(x) < l_1$. Moreover, for all $k \in \R$ the points $x \in (0,1)$ such that $x=kl(x)$ is a finite set and if $A= \{x\in (0,1): xl'(x) = l(x)\}$, then $\mu\{A\}=0$, 
where $\mu$ denotes the Lebesgue measure.
\end{itemize}
\par\medskip

This hypothesis contemplates the possibility that $h$ has a finite number of maxima and minima (see for instance Figure 2) or even the more degenerate situation where the function $h(x)$ has an infinite number of critical points (see Figure 3). For both cases, it does not seem possible to perform a change of variables that transform the problem in a domain with fixed period. 

\begin{figure}[h!]
\centering 
\subfigure[Function $h(x)=\frac{x}{l(x)}$]{\includegraphics[width=5cm, height=4cm]{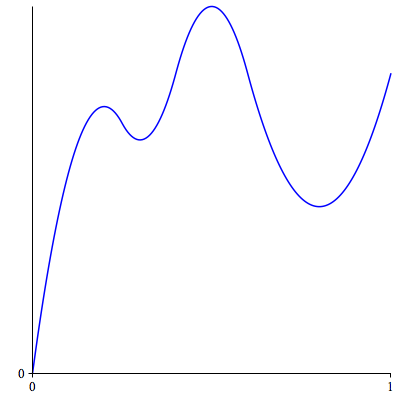}} \subfigure[Corresponding thin domain]{\includegraphics[width=8cm, height=3.5cm]{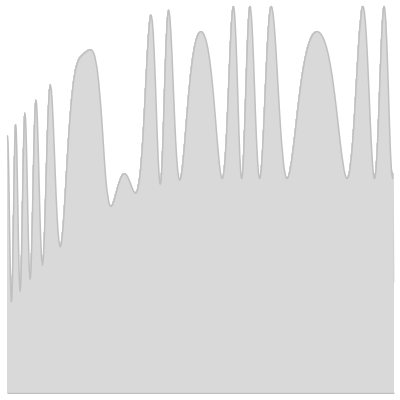}}
 \caption{$A$ is a finite set.}
\end{figure}

\begin{figure}[h!]
\centering 
\subfigure[Function $h(x)=\frac{x}{l(x)}$]{\includegraphics[width=5cm, height=4cm]{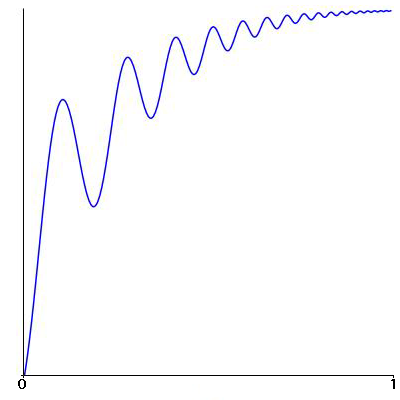}} \subfigure[Corresponding thin domain]{\includegraphics[width=8cm, height=3cm]{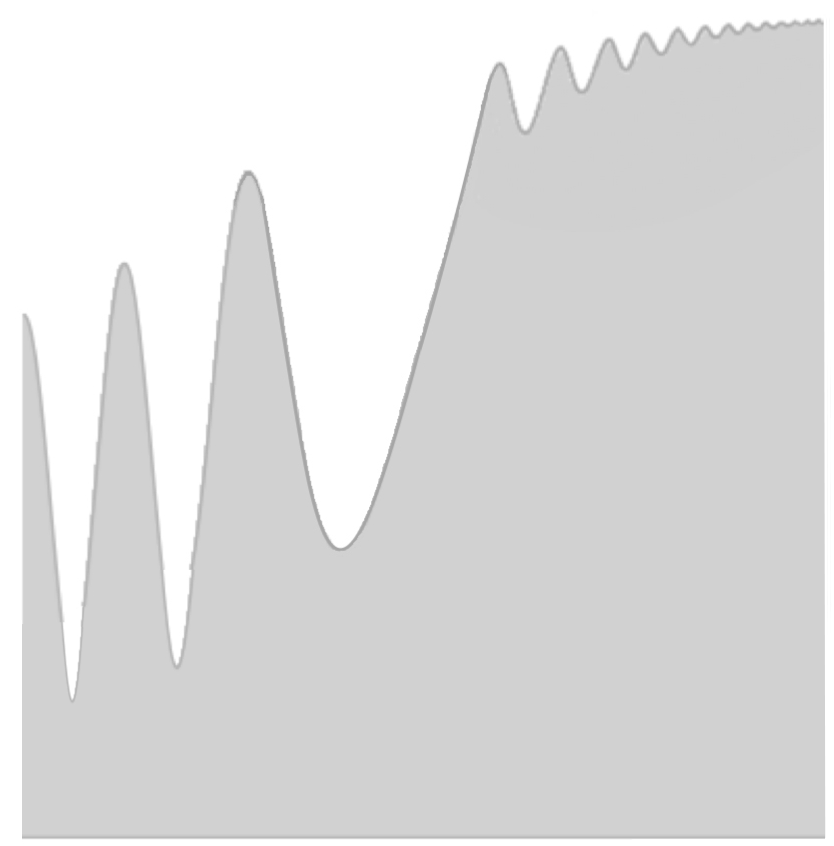}}
 \caption{$A$ is a infinite numerable set.}
\end{figure}

Our proposal consist in avoiding to make a change of variables and rather adapt the Unfolding  Operator method, which was initially devised to tackle purely periodic problems,  to this new ``locally periodic'' situation. We refer to   \cite{ArbDouHor,Cas} for the first descriptions of the unfolding operator method, \cite{CDG1,CDG2} for a systematic treatment of this method and \cite{DP} for an application of the method to a 2-dimensional domain with oscillating boundaries (actually \cite{DP} is the first time where the unfolding operator method is applied to a domain with an oscillatory boundary).

\par \medskip 

Let us mention that this adjustment of the method applies also to the case where the period does not depend on the spatial variable $x$ and we may consider this as a different method to obtain the results from \cite{ArrPer2011}. 
Moreover, the possibility to apply the unfolding operator method to a non-periodic situation express the robustness and power of the method.

Our  results, which were announced in \cite{ArrVil2014b}  for the simpler and more intuitive case where $h'(\cdot)>0$, will allow us to obtain the homogenized limit problem, together with a corrector result, for problems defined on thin domains with locally-periodic oscillatory boundary. This extension is performed for the case of thin domains, but the ideas and techniques can be applied to other situations like perforated domains where both hole and periodicity of the cell vary smoothly.   We will deal with this more general situation in a future publication.  As a matter of fact,
the limit problem we obtain is: 

$$
\left\{
\begin{gathered}
- \frac{l(x)}{|Y^*(x)|}\big(r(x) u_x\big)_x + u = f, \quad x \in (0,1) \\
 u'(0) = u' (1) = 0
\end{gathered}
\right.
$$
with $$r(x) =  \frac{1}{l(x)}\int_{Y^*(x)}\Big\{ 1 - \frac{\partial X(x)}{\partial y_1}(y_1,y_2) \Big\} dy_1 dy_2$$
where $Y^*(x)$ is the basic cell given by
$$
Y^*(x)= \{ (y_1,y_2) \in \R^2 \; : \; 0< y_1 < l(x), \quad 0 < y_2 < G(x,y_1) \},
$$
and $X(x)(\cdot,\cdot)$  is the unique solution of an appropriate PDE problem posed in the basic cell $Y^*(x)$.  Notice that if $l(x)$ is a constant independent of $x\in (0,1)$ we recover \eqref{GLP1}, \eqref{GLP2}, \eqref{GLP3}.


 
 
 The paper is organized as follow.   In Section \ref{Sec-notation}, we fix some notation that will be used throughout the paper. In \ref{Sec-PRE} we construct the unfolding operator  for locally periodic thin domains and show some important properties.   In Section \ref{Sec-Convergence}, we present some convergence results related to the unfolding operator. The essential result for the work is the compactness Theorem \ref{convergence prop}. 
 In Section \ref{Sec-Neumann}, we apply the previous results to obtain the homogenized limit problem. The Theorem \ref{limit problem}  shows two equivalent formulations
 of the homogenized limit problem. 
 Finally, in Section \ref{Sec-corrector}, we introduce an averaging operator $\mathcal{U}_\eps$ , the adjoint of the unfolding operator, and we prove its main properties. To end the paper, we use the averaging operator to obtain a corrector result.

\par\medskip\noindent {\bf Acknowledgment.} The authors are grateful to G. Griso, A. Damlamian and P. Donato for helpful comments and suggestions.

\section{Some notation}
\label{Sec-notation}
 Before embarking into the statements and proofs of the results let us clarify some notation that we will use throughout the article.


\par\medskip \noindent i) Observe that properly speaking there is not a basic cell associated to the domain $R^\eps$, since the periodicity properties vary from point to point in $x\in (0,1)$. Nevertheless, we will refer to the set $Y^*(x)$ defined by 
\begin{equation} \label{CELLL}
Y^*(x) = \{ (y_1,y_2) \in \R^2 \; : \; 0< y_1 < l(x), \quad 0 < y_2 < G(x,y_1) \}.
\end{equation}
as the basic cell at $x$. Notice that all these ``basic cells'' satisfy 
 $Y^*(x)\subset Y^*$ where $Y^*=(0,l_1)\times (0,G_1)$. 
\par\medskip\noindent ii)  The basic idea of the unfolding operator method in a purely periodic setting is to transform functions defined in $R^\eps$ into functions defined in a fixed domain.  With analogy to this case, we will consider the domain 
$$W = \{ (x, y_1, y_2) \in \R^3 \, : \; x \in (0, 1), \; (y_1, y_2) \in Y^*(x)\}$$ and $\hat W= \{ (x, y_1, y_2) \in \R^3 \, : \; x \in (0, 1), \; y_1\in \R$ and $0<y_2<G(x,y_1)\}$. Notice that $W=\{(x,y_1,y_2)\in \hat W:  0<y_1<l(x)\}$ and $W \subset (0,1) \times Y^*$.
Moreover, $W_{0} = \{(x,y_1) \in \R^2 : \; x \in (0, 1), \; y_1 \in (0, l(x)\}.$ Observe that this set is the bottom boundary of $W$.

\par\medskip\noindent iii)  Given an interval $(a, b) \subset \R$ we denote its characteristic function by $\chi_{(a, b)}$. $\chi^{\eps}$ is the characteristic function of $R^\eps$ and $\chi$ the characteristic function of $W$. Moreover, as it is usually done,  $\, \widetilde{}\;$ is the standard extension by zero operator.

\par\medskip\noindent iv) We will use the subindex $\#$ to denote periodicity with respect to the $y_1$ variable in the following sense.  For every fixed $x \in (0, 1)$,  the space $C^k_{\#}\big(Y^*(x)\big)$ consists of all functions $\varphi$ which are obtained as restrictions to $Y^*(x)$ of functions in $C^k(\mathbb{R}^2)$ which are  $l(x)$-periodic in the first variable. That is
$$C^k_{\#}\big(Y^*(x)\big)=\{ \varphi_{| Y^*(x)}:  \varphi\in C^k(\mathbb{R}^2), \, \varphi(y_1+l(x),y_2)=\varphi(y_1,y_2), \quad \forall (y_1,y_2)\in \mathbb{R}^2\}.$$
This is a Banach space with the usual norm $\|\cdot\|_{C^k(Y^*(x))}$.  

Also,  $W^{1,p}_{\#}\big(Y^*(x)\big)$ is the completion for the norm of $W^{1,p}\big(Y^*(x)\big)$ of  $C^{\infty}_{\#}\big(Y^*(x)\big)$.  Moreover, the space $C^\infty_\#(W)$ is the space of functions which are restrictions to $W$ of functions $\varphi\in C^\infty(\mathbb{R}^3)$ which for fixed $x\in (0,1)$ they are periodic in the $y_1$ variable of period $l(x)$, that is:  $\varphi(x,y_1+l(x),y_2)=\varphi(x,y_1,y_2)$. 

\par\medskip\noindent v) 
We define the spaces of Banach space-valued functions $L^p\Big((0,1); W^{1,q}_\#\big(Y^*(x)\big)\Big)$, \\
$L^p\Big((0,1); C^k_{\#}\big({Y^*(x)}\big)\Big)$  and $W^{1,p}\Big((0,1); C^k_{\#}\big({Y^*(x)}\big)\Big)$ in the following way: 

\begin{itemize}
\item[(1)] The space $L^p\Big((0,1); W^{1,q}_{\#}\big(Y^*(x)\big)\Big)$ consists of the measurable functions $\varphi: W \to \R$ such that $\varphi(x) \in W^{1,q}_{\#}\big(Y^*(x)\big)$ a.e. $x\in (0,1)$, with
\begin{equation*}
\|\varphi\|_{L^p\Big((0,1); W^{1,q}_{\#}\big(Y^*(x)\big)\Big)}:=     
\left\{
\begin{gathered}
\Big(\int_{0}^{1} \|\varphi(x)\|^p_{W^{1,q}_{\#}\big(Y^*(x)\big)}\; dx\Big)^{1/p} < \infty,  \hbox{ for } 1\leq p < \infty\\
\operatorname*{ess \,sup}_{x \in (0,1)}\|\varphi(x)\|_{W^{1,q}_{\#}\big(Y^*(x)\big)}< \infty,  \hspace*{0.3cm} \hbox{ for }  p = \infty.
\end{gathered}
\right.
\end{equation*}

Notice that $L^2\Big((0,1); H^1_{\#}\big(Y^*(x)\big)\Big)$ actually coincides with the space of functions
$\varphi \in L^2(W)$ such that $\frac{\partial \varphi}{\partial y_1}$, $\frac{\partial \varphi}{\partial y_2}$ belong to $L^2(W)$ and 
$\varphi(x, \cdot, \cdot)$ is $l(x)$-periodic in the first variable $y_1$.
\item[(2)] The space $L^p\Big((0,1); C^k_{\#}\big ({Y^*(x)}\big)\Big)$ comprises all strongly measurable functions $\varphi: W \to \R$  such that $\varphi(x) \in C^k_{\#}\big(\; {Y^*(x)}\; \big)$ and 
\begin{equation*}
\|\varphi\|_{L^p\Big((0,1); C^k_{\#}\big (\, {Y^*(x)}\, \big)\Big)}:=     
\left\{
\begin{gathered}
\Big(\int_{0}^{1} \|\varphi(x)\|^p_{C^k_{\#}\big ({Y^*(x)}\big)}\; dx\Big)^{1/p} < \infty,  \hbox{ for } 1\leq p < \infty\\
\operatorname*{ess \,sup}_{x \in (0,1)}\|\varphi(x)\|_{C^k_{\#}\big ({Y^*(x)}\big)}< \infty,  \hspace*{0.3cm}  \hbox{ for }  p = \infty.
\end{gathered}
\right.
\end{equation*}

\item[(3)]  The space $W^{1,p}\Big((0,1); C^k_{\#}\big( \, {Y^*(x)}\, \big)\Big)$ consists of all functions $\varphi \in L^p\Big((0,1); C^k_{\#}\big(\, {Y^*(x)}\,\big)\Big)$ such that $\partial_x\varphi$ exists in the weak sense and belongs to $L^p\Big((0,1); C^k_{\#}\big(\,{Y^*(x)}\,\big)\Big)$. Furthermore,
{\small
\begin{equation*}
\|\varphi\|_{W^{1,p}\Big((0,1); C^k_{\#}\big (\, {Y^*(x)}\,\big)\Big)}:=     
\left\{
\begin{gathered}
\Big(\int_{0}^{1} \|\varphi(x)\|^p_{C^k_{\#}\big ({Y^*(x)}\big)} +  \| \partial_x\varphi(x)\|^p_{C^k_{\#}\big ({Y^*(x)}\big)}\; dx\Big)^{1/p} (1\leq p < \infty)\\
\operatorname*{ess \,sup}_{x \in (0,1)}\Big(\|\varphi(x)\|_{C^k_{\#}\big ({Y^*(x)}\big)} +  \| \partial_x\varphi(x)\|_{C^k_{\#}\big ({Y^*(x)}\big)}\Big)\, (p = \infty).
\end{gathered}
\right.
\end{equation*}}
We usually write $H^1\Big((0,1); C^k_{\#}\big( \,{Y^*(x)}\, \big)\Big)=W^{1,2}\Big((0,1); C^k_{\#}\big( \,{Y^*(x)}\, \big)\Big)$.
\end{itemize}

\section{The unfolding operator in a  domain with locally-periodic oscillatory boundary } \label{Sec-PRE}

In this section we construct the unfolding operator for the locally periodic case and show some basic properties.  
Due to the lost of periodicity, 
a delicate point in this construction is to define an appropriate partition of the limit segment  $I=(0,1)$ which will be in accordance with the oscillatory behavior of the thin domain 
(\ref{Def-thindomain}). We will start defining the concept of ``admissible partition'' of the interval $(0,1)$. 
%
\begin{definition}\label{partition0}
An admissible partition is given by the family of ordered numbers $\{x_k^\eps\}_{k=0}^{N_\eps+1}$ for all $0<\eps\leq \eps_0$, satisfying 
$$0=x_0^\eps<x_1^\eps<\ldots<x_{N_\eps}^\eps<x_{N_\eps+1}^\eps=1$$
Moreover, for almost every point $x\in (0,1)$ there exist $0<\eps_1\leq \eps_0$ and a constant $C=C(x)$ such that for all $ 0<\eps<\eps_1\leq \eps_0$ there is a point $x_k^\eps$ of the partition which satisfies $|x- x_k^\eps|\leq  C\eps$.

We will refer to the partition as $\{x_k^\eps\}$.
\end{definition}

Some of the results below will be proved for a general admissible partition $\{x_k^\eps\}$.

We consider now a general admissible partition $\{x_k^\eps\}$.  For every $x \in (0,1)$ there exists a unique element of the partition, $x_k^\eps$, such that $x \in [x_{k}^\eps, x_{k+1}^\eps)$. By the analogy to the periodic case, we denote this point $x_k^\eps$ 
by $[x]_\eps.$ In addition, since the partition is not equally spaced we consider  for every $x \in (x_{k}^\eps, x_{k+1}^\eps)$ the factor $\Gamma_{[x]_\eps}$ given by
$$\Gamma_{[x]_\eps} := \frac{x_{k+1}^\eps - x_{k}^\eps}{l(x_{k}^\eps)}.$$ 
Then, for each $x \in (0, 1)$ there is a unique point in $y_1\in \big(0, l([x]_\eps)\big)$ such that
$$ x =  [x]_\eps +\Gamma_{[x]_\eps} y_1.$$

We are in a position to define the Unfolding Operator in our setting. 

\begin{definition}\label{unfolding}
Let $\{x_k^\eps\}$ be a general admissible partition as in Definition \ref{partition0}.  Let $\varphi$ be a  Lebesgue-measurable function in $R^{\eps}$. We define the unfolding operator $\mathcal{T_\eps}$ associated to the partition $\{x_k^\eps\}$,  acting on $\varphi$, as the function $\mathcal{T}_\eps(\varphi)$ defined in $(0,1)\times Y^*$ as:
\begin{equation*}
\mathcal{T}_\eps(\varphi) (x, y_1, y_2)=
\left\{
\begin{gathered}
 \widetilde{\varphi} \Big( [x]_\eps +\Gamma_{[x]_\eps} y_1, \eps y_2\Big) \quad \hbox{for} \quad y_1 \in \big(0, l([x]_\eps)\big),\\
 0     \hspace*{3.6cm} \hbox{for}  \quad y_1 \in \big(l([x]_\eps), l_1\big).
\end{gathered}
 \right.
\end{equation*}
\end{definition}

As in classical periodic homogenization, the unfolding operator reflects two scales : the ``macroscopic'' scale $x$ gives the position in the interval $(0,1)$ 
and the ``microscopic'' scale $(y_1, y_2)$ gives the position in the cell $Y^*$. However, due to the locally periodic oscillations of the domain $R^\eps$, 
the definition given here differs from the introduced in periodic cases. In this case, we do not have a fixed cell that describes the domain
 $R^\eps$, therefore, in the definition we use the rectangle $Y^*= (0,l_1) \times (0,G_1)$, the extension by zero and the factors $\Gamma_{[\cdot]_\eps}$ to cover $R^\eps$ and to reflect the oscillations and the variable period. As 
 a consequence we remark that there exist two crucial differences:
 
 \begin{enumerate}
 \item The support of the functions $\mathcal{T}_\eps(\varphi)$ depends on $\eps$. (See Figure \ref{domain3}). As a matter of fact for a general $\varphi$ the support of the function $\mathcal{T}_\eps(\varphi)$ is given by
\begin{eqnarray}\label{W-eps}
W^\eps=\Big\{(x,y_1,y_2): x\in (0,1), 0< y_1< l([x]_\eps), 0<y_2< G \Big(  [x]_\eps +\Gamma_{[x]_\eps} y_1, \frac{1}{\eps}\big([x]_\eps +\Gamma_{[x]_\eps} y_1\big)\Big) \Big\}\nonumber \\
= \bigcup_{k=0}^{N_\eps} [x_k^\eps, x_{k+1}^\eps) \times Y^\eps_k\subset [0,1]\times Y^*.\hspace*{8.15cm}
\end{eqnarray}
 where 
 $$Y_k^\eps= \big\{ (y_1,y_2):  0< y_1< l(x_k^\eps), 0<y_2< G \big(  x_k^\eps +\Gamma_{x_k^\eps} y_1, \frac{1}{\eps}\big(x_k^\eps +\Gamma_{x_k^\eps} y_1\big)\big) \big\},\; k=0,1,\ldots, N_\eps.$$
 Then, we have to prove that the sequence of these three dimensional domains $W^\eps$ 
converges in certain sense to the fixed domain $W$ as $\eps$ goes to zero. (See Proposition (\ref{domainunfolding}) below).
\item Even if $\varphi$ is very regular, $\mathcal{T_\eps(\varphi)}$ does not inherit regularity  as a function of $(y_1, y_2)$. This is a delicate point to obtain the limit of $\mathcal{T_\eps}(\frac{\partial\varphi}{x_i}),$
$i=1,2$. To overcome this difficulty we will use an approach inspired by \cite{AL}.
 \end{enumerate}

 \begin{figure}[H]
  \centering
    \includegraphics [width=0.5\textwidth]{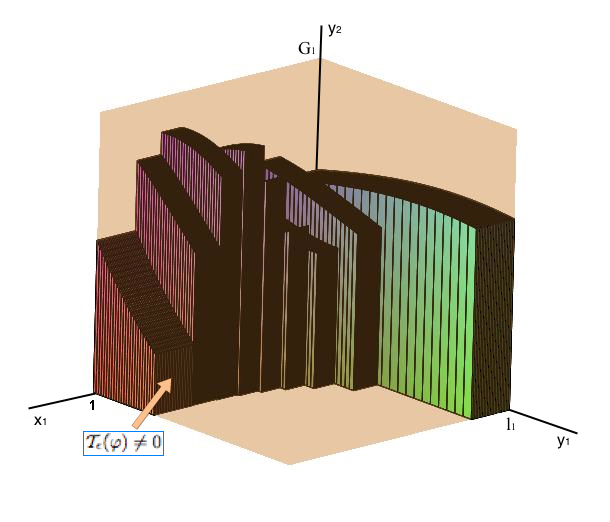}
    \caption{The set $W^\eps$, the support of $\mathcal{T}_\eps(\varphi)$. }
    \label{domain3}
\end{figure}

 In the following proposition we show the main properties of unfolding operator which will be used. Some of them are straightforward  and their proofs are omitted.
\begin{proposition}\label{properties}
The unfolding operator $\mathcal{T_\eps}$ associated to a general admissible partition $\{x_k^\eps\}$  has the following properties:
\begin{enumerate}
\item[i)] $ \mathcal{T_\eps}$ is a linear operator.
\item[ii)] $\mathcal{T}_\eps(\varphi \psi)=\mathcal{T_\eps(\varphi)} \mathcal{T_\eps(\psi)}$ and $\mathcal{T}_\eps( f \circ \psi) = f \circ \mathcal{T_\eps(\psi)},$ $\, \forall \, \varphi, \psi$ Lebesgue-measurable functions in $R^\eps$ and $f: \R \to \R$ a continuous function with $f(0)=0.$
\item[iii)] Unfolding criterion for integrals $($u.c.i.$)$ :
\begin{equation}\label{u.c.i}
\begin{gathered}
\int_{ (0,1) \times Y^*}\frac{1}{l([x]_\eps)} \mathcal{T_\eps(\varphi)} (x, y_1, y_2) dx dy_1 dy_2 = \frac{1}{\eps}\int_{R ^\eps} \varphi (x, y) dx dy, \quad \forall \, \varphi \in L^1(R^\eps),\\
\frac{1}{\eps} \int_{R ^\eps} l([x]_\eps)\varphi (x, y) dx dy  = \int_{ (0,1) \times Y^*} \mathcal{T_\eps(\varphi)} (x, y_1, y_2) dx dy_1 dy_2, \quad \forall \, \varphi \in L^1(R^\eps).
 \end{gathered}
\end{equation}
\item[iv)] For every $\varphi \in L^p(R^\eps)$, $\mathcal{T_\eps(\varphi)} \in L^p\big( (0,1)\times Y^*\big)$, with $1\leq p < \infty$. In addition, the following relationship exists between their norms:
$$\|\mathcal{T_\eps(\varphi)}\|_{L^p\big( (0,1)\times Y^*\big)}\leq \Big(\frac{l_1}{\eps}\Big)^{\frac{1}{p}} \, \|\varphi\|_{L^p(R^\eps)}, $$
$$ \Big(\frac{l_0}{\eps}\Big)^{\frac{1}{p}} \|\varphi\|_{L^p(R^\eps)} \leq \|\mathcal{T_\eps(\varphi)}\|_{L^p\big( (0,1)\times Y^*\big)}.$$
In the special case $p=\infty$, 
$$\|\mathcal{T_\eps(\varphi)}\|_{L^\infty \big( (0,1)\times Y^*\big)}=  \|\varphi\|_{L^\infty(R^\eps)}.$$
\item[v)]  Let $\psi \in  C^\infty_\#(W) $. 
 We define the sequence $\{\psi^\eps\}$ by
$$\psi^\eps(x, y) = \psi\big(x, \frac{x}{\eps}, \frac{y}{\eps}) \quad \forall (x, y) \in R^\eps,$$
then, $ \psi^\eps \in C^\infty( \overline{R^\eps})$ and
$$\mathcal{T_\eps(\psi^\eps)} (x, y_1, y_2) =  \widetilde{\psi} \Big( [x]_\eps +\Gamma_{[x]_\eps} y_1, \frac{ [x]_\eps +\Gamma_{[x]_\eps} y_1}{\eps}, y_2\Big) \chi^\eps\Big( [x]_\eps +\Gamma_{[x]_\eps} y_1, \eps y_2\Big) \chi_{\big(0,l([x]_\eps)\big)}(y_1),$$ 
for all $ (x, y_1, y_2) \in (0, 1) \times Y^*.$
\end{enumerate}

\end{proposition}

\begin{proof}
\begin{enumerate}
\item[$i)$]  Immediate from the definition of the unfolding operator.
 \item[$ii)$] Simple consequence of definition of the unfolding operator.
 \item[$iii)$] Let $\varphi \in L^1(R^\eps)$. The proof is similar for both equalities so that we will see only the first one:
\begin{eqnarray*}
\lefteqn{ \int_{(0, 1)\times Y^*} \frac{1}{l([x]_\eps)} \mathcal{T_\eps(\varphi)}  (x, y_1, y_2)\;  dx dy_1 dy_2}\\
 & = & \sum_{k=0}^{N_\eps}\int_{(x_{k}^\eps, x_{k+1}^\eps) \times Y^*} \frac{1}{l(x_{k}^\eps)}\widetilde{\varphi} \Big(x_{k}^\eps +\Gamma_{x_{k}^\eps} y_1, \eps y_2\Big)\chi_{(0,l(x_{k}^\eps))}(y_1) \;dx dy_1 dy_2 \\
 & = & \sum_{k=0}^{N_\eps} (x_{k+1}^\eps - x_{k}^\eps) \int_{\big(0, l(x_{k}^\eps)\big) \times \big(0, G_1\big)} \frac{1}{l(x_{k}^\eps)}\widetilde{\varphi} \Big(x_{k}^\eps +\Gamma_{x_{k}^\eps} y_1, \eps y_2\Big) \;  dy_1 dy_2 \\
& = & \sum_{k=0}^{N_\eps}\frac{1}{\eps}\int_{(x_{k}^\eps, x_{k+1}^\eps) \times (0, \eps G_1)} \widetilde{\varphi}(x, y)\; dx dy\\
& = & \frac{1}{\eps} \int_{(0, 1) \times (0, \eps G_1)}  \widetilde{\varphi}(x, y) \;dy dx = \frac{1}{\eps}\int_{R ^\eps} \varphi (x, y) \;dx dy.
 \end{eqnarray*}
 \item[$iv)$] Consequence of  iii). 
 \item[$v)$] First we will see that $\{\psi^\eps\}$ is well defined. For every $x\in (0,1)$ there exists $k \in \N$ large enough so that $\frac{x}{\eps} - kl(x) \in (0, l(x)).$ Furthermore,  since $(x, y) \in R^\eps$ we have that $0 <y < \eps G\big(x, \frac{x}{\eps}\big) = G\big(x, \frac{x}{\eps} - kl(x)\big)$. Hence, $\big(\frac{x}{\eps} - kl(x), \frac {y}{\eps}) \in Y^*(x)$ and we can conclude that the function is correctly defined.

 Now, the proof is immediate from the regularity of $\psi$ and the Definition \ref{unfolding}. 

 \end{enumerate} 
\end{proof}

\begin{remark}\label{important}
Properties iii) and v) will be very important when obtaining  the homogenized limit problem.  As a matter of fact, property iii) (u.c.i property), shows that the unfolding operator preserves the integral of the functions (up to the multiplicative piecewise constant function $\frac{\eps}{l([\cdot]_{\eps})}$) and property v) shows the relationship between test functions. 
\end{remark}

\begin{remark}\label{norm}
 Due to the order of the height of the thin domain the factor $\frac{1}{\eps}$ appears in the criterion for integrals and in property iv). Therefore, it makes sense to
 consider the following measure in the thin domain
$$\rho_{\eps}(\mathcal{O}) = \frac{1}{\eps}|\mathcal{O}|, \; \forall \mathcal{O} \subset R^\eps.$$
This measure allows us to  introduce the spaces $L^p( R^\eps, \rho_\eps)$ and  $W^{1,p}( R^\eps, \rho_\eps)$, for $1\leq p < \infty$
endowed with the norms obtained rescaling the usual norms by the factor $\frac{1}{\eps}$, that is, 

$$|||\varphi|||_{L^p(R^\eps)} = \eps^{-1/p}||\varphi||_{L^p(R^\eps)},  \quad \forall \varphi \in L^p(R^\eps),$$
$$|||\varphi|||_{W^{1,p}( R^\eps)} = \eps^{-1/p}||\varphi||_{W^{1,p}( R^\eps)}, \quad \forall \varphi \in W^{1,p}(R^\eps).$$

Then, using the notation above we can rewrite the property iv) as
$$\|\mathcal{T_\eps(\varphi)}\|_{L^p\big( (0,1)\times Y^*\big)}\leq {l_1}^{1/p} \, |||\varphi|||_{L^p(R^\eps)}, $$
$$ |||\varphi|||_{L^p(R^\eps)} \leq \frac{1}{{l_0}^{1/p}}\|\mathcal{T_\eps(\varphi)}\|_{L^p\big( (0,1)\times Y^*\big)}.$$

\end{remark}


\section{Convergence properties of the unfolding operator}\label{Sec-Convergence} 

In this section we analyze the convergence properties of the unfolding operator defined in the previous section, as $\eps \rightarrow 0$. To have good convergence properties of the unfolding operator we will need to choose an appropriate admissible partition which is related to the variable period $l(x)$ of the boundary of the thin domain and we will denote this special partition as the $l(x)$-partition.  
 To construct the $l(x)$-partition, first of all, notice that from hypothesis  {\bf (H)} the function $x\to \frac{x}{l(x)}$ has the set of critical points of measure zero and the inverse image of 
 every point is at most a finite set.
 
 \begin{definition}\label{partition}
 For every $\eps>0$ fixed, we define $M^\eps$ the largest integer such that $ \displaystyle M^\eps \eps \leq \max_{x\in [0,1]}\Big\{\frac{x}{l(x)}\Big\}$ and consider  the points $x\in (0,1)$ such that $\frac{x}{l(x)}= k\eps$, for all $k=1,2,\ldots, M^\eps$ (see Figure 2).  Hence,  the $l(x)$-partition is defined by $\{x_i^\eps\}$, $i=0,1, \ldots, N^\eps+1$ such that 
\begin{itemize}
\item[$\bullet$] $x^\eps_0=0$ and $x_{N^\eps+1}^\eps=1.$
\item[$\bullet$] Given $x_i^\eps$ a point of the partition, $i\neq N^\eps+1$, there exists $k \in \{0, \dots, M^\eps\}$ such that $\frac{x_i^\eps}{l(x_i^\eps)}= k\eps$.
\item[$\bullet$] Two consecutive points of the partition satisfies $$ \Big|\frac{x_{i+1}}{l(x_{i+1})} - \frac{x_i}{l(x_i)}\Big|=\eps, \hbox{  or }\quad \Big|\frac{x_{i+1}}{l(x_{i+1})} - \frac{x_i}{l(x_i)}\Big|=0$$
\end{itemize}
 \end{definition}

\begin{figure}[H]
  \centering{\includegraphics [width=11cm, height=5cm]{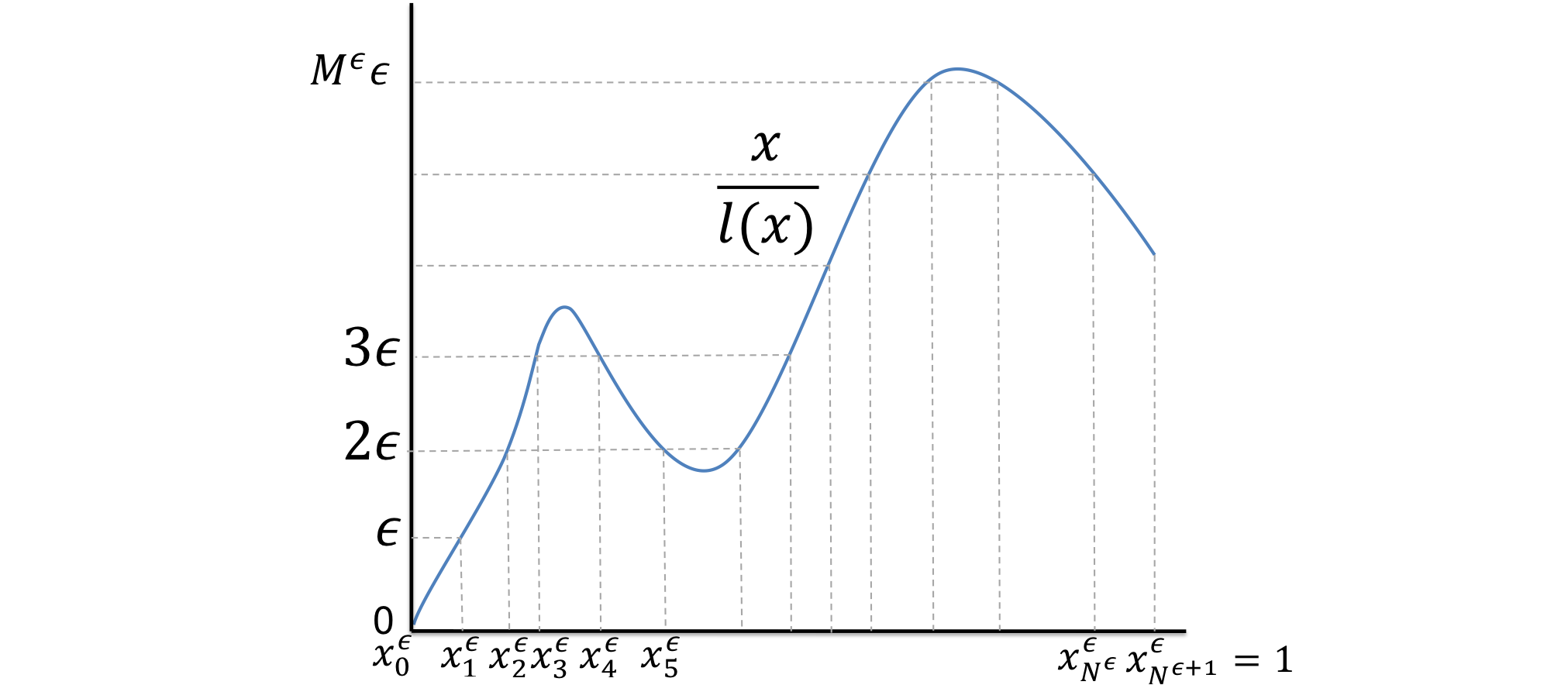}
    \caption{The $l(x)$-partition $\{x_k^\eps\}$}}
    \label{thin}
\end{figure}

The $l(x)$-partition is correctly defined and it is an admissible partition by assumption  { \bf (H)}. This hypothesis guarantees that inverse image of every $y\in \R$ by the function $\frac{x}{l(x)}$ is at most a finite set and, as a consequence, the $l(x)$-partition has a finite number of points. 
Moreover, notice that the distance between two 
consecutive points is not constant but somehow reproduces the locally periodic structure of the thin domain and satisfies that for every $x\in(0,1) \setminus A$ there exist
$\eps_1>0$ and a constant $C$ such that $0\leq x-[x]_\eps<C\eps, \hbox{ for }  0<\eps<\eps_1.$

\begin{remark}
To ``justify ''  the choice of the $l(x)$-partition, consider the particular case where the function which defines the oscillatory boundary of the thin domain is given by
$$G(x,y) = 2 + \cos\Big(\frac{2\pi y}{l(x)}\Big)$$
with $l(\cdot)$ verifying the assumption {\bf (H)}.
If we look at the points which are at the top part of the oscillatory boundary, that is, $G(x,x/\eps) =3$, or equivalently,   
$$ \cos \Big(\frac{2\pi x}{\eps l(x)}\Big) =1,$$
 we obtain a sequence of points which verify
$$\frac{2\pi x}{\eps l(x)}=2\pi k, \; k=0,1,2\dots$$
Observe that these points coincide with the points of the $l(x)$-partition. It turns out that these points also work when the amplitude of the oscillations also varies in space since they continue 
to reproduce in a good way the locally oscillatory behavior of the oscillations.
\end{remark}

\begin{remark}
Observe that hypothesis { \bf (H)} is a natural assumption to define the $l(x)-$partition. In addition, it allows us to deal with a large kind of period functions. For example, we may consider the case where the function $h(x)=\frac{x}{l(x)}$ has a finite number of critical points which generates some strange oscillations in the thin domain, (see Figure 2 ). 

Moreover, the behavior of the oscillatory boundary is more pathological if we consider for instance a period $l(\cdot)$ such that $h(x)$ has a countably infinite set of critical points, (see Figure 3).
\end{remark}

\par\medskip

We start showing the following key property of the partition $\{x_k^\eps\}$  chosen in this Section. 

\begin{proposition}\label{Convergence-of-G}
If $\{x_k^\eps\}$ is the $l(x)$-partition and with the notation above, we have the following
\begin{equation}\label{convergence 2}
G\big(  [x]_\eps +\Gamma_{[x]_\eps} y_1, \frac{1}{\eps}( [x]_\eps  +\Gamma_{[x]_\eps} y_1)\big) - G(x, y_1)\stackrel{\eps\to 0}\longrightarrow 0 \; \hbox{ a.e  } (x,y_1)\in W_0.
\end{equation}
\end{proposition}
\begin{proof}
In order to show (\ref{convergence 2}), we observe that since $G(x, \cdot)$ is $l(x)$-periodic and all the points  constructed in Definition \ref{partition} satisfy $\frac{[x]_\eps}{\eps l([x]_\eps)} = k^\eps \in \N$ we have
$$G\Big(  [x]_\eps +\Gamma_{[x]_\eps} y_1, \frac{1}{\eps}\big( [x]_\eps  +\Gamma_{[x]_\eps} y_1\big)\Big) = G\Big(  [x]_\eps +\Gamma_{[x]_\eps} y_1, \frac{1}{\eps}\big( [x]_\eps - l\Big([x]_\eps +\Gamma_{[x]_\eps} y_1\Big)\frac{[x]_\eps}{l([x]_\eps)} +\Gamma_{[x]_\eps} y_1\big)\Big). $$

Then, using the regularity properties of $G(\cdot, \cdot)$ and $l(\cdot)$, see \textbf{(H)}, we only have to prove the following convergences for almost every $x \in (0,1)$ and $y_1 \in (0, l(x))$:
\begin{equation}\label{1}
[x]_\eps \stackrel{\eps\to 0}\longrightarrow x \; \hbox{a.e } x \in (0,1).
\end{equation}

\begin{equation}\label{1variable}
[x]_\eps +\Gamma_{[x]_\eps} y_1 \stackrel{\eps\to 0}\longrightarrow x,  \; \hbox{a.e  }(x, y_1) \in W_0.
\end{equation}

\begin{equation}\label{2variable}
\frac{1}{\eps}\Big( [x]_\eps - l\Big([x]_\eps +\Gamma_{[x]_\eps} y_1\Big)\frac{[x]_\eps}{l([x]_\eps)} +\Gamma_{[x]_\eps} y_1\Big)\stackrel{\eps\to 0}\longrightarrow y_1,   \; \hbox{a.e } (x, y_1) \in W_0.
\end{equation}
Since the hypothesis \textbf{(H)} guarantees that the set A has null measure it is enough to prove the convegences of the Proposition for all $x\in (0,1)\setminus A$.

Let $x$ be a point in $(0,1)\setminus A$, from the definition of the $l(x)-$partition we know that there exist $\eps_1>0$ and a constant $C$ such that $x\in (x_k^\eps, x_{k+1}^\eps)$ and $0<x_{k+1}^\eps-x_k^\eps<C\eps,$ for some $k \in \{0,1,2, \dots, N^\eps+1\}$ and $\forall \, 0<\eps<\eps_1.$
Therefore, the convergences (\ref{1}) and (\ref{1variable}) are immediate from the definition of $[x]_\eps $ and $\Gamma_{[x]_\eps}$.

In order to prove (\ref{2variable}) we assume that $x \in (0,1)\setminus A$. From the definition of the $l(x)-$partition we know that for $\eps$ small enough there exists $p^\eps \in \{0,1,\dots,M^\eps\}$ such that $[x]_\eps = x^\eps_{k} = \eps p^\eps l(x^\eps_{k})$ and $x^\eps_{k+1}= \eps (p^\eps+1) l(x^\eps_{k+1})$, the function $\frac{x}{l(x)}$ increases in the small interval $(x^\eps_{k}, x^\eps_{k+1})$, or    $x^\eps_{k+1}= \eps (p^\eps-1) l(x^\eps_{k+1})$ in case $\frac{x}{l(x)}$ decreases. We suppose that $x^\eps_{k+1}= \eps (p^\eps+1) l(x^\eps_{k+1})$, the proof is similar in the other case. 

We study the limit of the following two terms: 
\par\bigskip\noindent {\bf  (a). Limit of $\frac{1}{\eps}\Gamma_{[x]_\eps} y_1$. }

Since $x^\eps_{k+1} =  \eps (p^\eps+1) l(x^\eps_{k+1})$ and $x^\eps_{k} =  \eps p^\eps l(x^\eps_{k})$ we have
\begin{eqnarray*}
\lefteqn{\frac{1}{\eps}\Gamma_{[x]_\eps} y_1 =  \frac{x^\eps_{k+1} - x^\eps_{k}}{\eps l(x^\eps_{k})}y_1= p^\eps \frac{y_1}{l(x^\eps_{k})} \big(l(x^\eps_{k+1}) - l(x^\eps_{k})\big) + \frac{y_1}{l(x^\eps_{k})}l(x^\eps_{k+1})}\\
& = &  \frac{x^\eps_{k}}{\eps l(x^\eps_{k})^2}y_1\big(l(x^\eps_{k+1}) - l(x^\eps_{k})\big) +  \frac{y_1}{l(x^\eps_{k})}l(x^\eps_{k+1}) \\
& = &  \frac{x^\eps_{k}}{ l(x^\eps_{k})}y_1 \frac{l(x^\eps_{k+1}) - l(x^\eps_{k})}{x^\eps_{k+1} - x^\eps_{k}} \frac{x^\eps_{k+1} - x^\eps_{k}}{\eps l(x^\eps_{k})} + \frac{y_1}{l(x^\eps_{k})}l(x^\eps_{k+1})\\
& = & \frac{x^\eps_{k}}{ l(x^\eps_{k})} \frac{l(x^\eps_{k+1}) - l(x^\eps_{k})}{x^\eps_{k+1} - x^\eps_{k}}\frac{1}{\eps}\Gamma_{[x]_\eps} y_1  + \frac{y_1}{l(x^\eps_{k})}l(x^\eps_{k+1}).
\end{eqnarray*}

Consequently we obtain
\begin{equation*}
\Big(1  - \frac{x^\eps_{k}}{ l(x^\eps_{k})} \frac{l(x^\eps_{k+1}) - l(x^\eps_{k})}{x^\eps_{k+1} - x^\eps_{k}}\Big)\frac{1}{\eps}\Gamma_{[x]_\eps} y_1 = \frac{y_1}{l(x^\eps_{k})}l(x^\eps_{k+1}).
\end{equation*}

Moreover,  we can prove that
 $$ 1  - \frac{x^\eps_{k}}{ l(x^\eps_{k})} \frac{l(x^\eps_{k+1}) - l(x^\eps_{k})}{x^\eps_{k+1} - x^\eps_{k}}\neq 0 \quad \text{for $\eps$ sufficiently small}.$$

  If we suppose that this is not true and we pass to the limit then we have , at least for a subsequence, that
 \begin{equation*}
l'(x) = \frac{l(x)}{x}. 
\end{equation*} 
This contradicts the fact that $x \not \in A$.

Hence we can write
 \begin{equation}\label{gamma}
\frac{1}{\eps}\Gamma_{[x]_\eps} y_1 = \Big(1  - \frac{x^\eps_{k}}{ l(x^\eps_{k})} \frac{l(x^\eps_{k+1}) - l(x^\eps_{k})}{x^\eps_{k+1} - x^\eps_{k}}\Big)^{-1}\frac{y_1}{l(x^\eps_{k})}l(x^\eps_{k+1}).
\end{equation}

Taking into account the above equality (\ref{gamma}), the hypothesis \textbf{(H)} and $x \not \in A$ we get the limit:
\begin{equation}\label{limit gamma}
\frac{1}{\eps}\Gamma_{[x]_\eps} y_1 \stackrel{\eps\to 0}\longrightarrow \Big( 1- \frac{x}{l(x)}l'(x)\Big)^{-1}y_1.
\end{equation}

\par\bigskip\noindent {\bf  (b). Limit of $\frac{1}{\eps}\Big( [x]_\eps - l\big([x]_\eps +\Gamma_{[x]_\eps} y_1\big)\frac{[x]_\eps}{l([x]_\eps)}\Big)$. }
\begin{eqnarray*}
\frac{1}{\eps}\Big( [x]_\eps - l([x]_\eps +\Gamma_{[x]_\eps} y_1)\frac{[x]_\eps}{l([x]_\eps)}\Big) =  - \frac{x^\eps_{k}}{  l(x^\eps_{k})} \frac{l(x^\eps_{k} +\Gamma_{x^\eps_{k}} y_1) - l(x^\eps_{k})} 
{\Gamma_{x^\eps_{k}} y_1} \frac{\Gamma_{x^\eps_{k}} y_1}{\eps}
\end{eqnarray*}

Now it is possible to pass to the limit in the right-hand side by using the convergence (\ref{limit gamma}). Hence, we have that
\begin{equation}\label{limit b}
\frac{1}{\eps}\Big( [x]_\eps - l([x]_\eps +\Gamma_{[x]_\eps} y_1)\frac{[x]_\eps}{l([x]_\eps)}\Big)\stackrel{\eps\to 0}\longrightarrow - \frac{x}{l(x)}l'(x) \Big( 1- \frac{x}{l(x)}l'(x)\Big)^{-1}y_1.
\end{equation}

Therefore, from (\ref{limit gamma}) and (\ref{limit b}), we obtain the converge (\ref{2variable}) 
$$\frac{1}{\eps}\big( [x]_\eps - l([x]_\eps +\Gamma_{[x]_\eps} y_1)\frac{[x]_\eps}{l([x]_\eps)} +\Gamma_{[x]_\eps} y_1\big)\stackrel{\eps\to 0}\longrightarrow - \frac{x}{l(x)}l'(x) \Big( 1- \frac{x}{l(x)}l'(x)\Big)^{-1}y_1
+  \Big( 1- \frac{x}{l(x)}l'(x)\Big)^{-1}y_1 = y_1. $$
This concludes the proof of the proposition. 
\end{proof}

 We show now that the domains $W^\eps$, see \eqref{W-eps}, converge to the domain $W$ in the sense that the characteristic functions converge strongly in $L^p$, $1\leq p < \infty$.  Let us denote by $\chi$ the characteristic function of $W$ and by $\chi^\eps$ is the characteristic function of $R^\eps$. Therefore, $T_\eps(\chi^\eps)$ is the characteristic function of $W^\eps$. The fact that $W^\eps$ `` approaches '' $W$  is expressed in the following result.

\begin{proposition}\label{domainunfolding}
With the notations above and if $\mathcal{T}_\eps$ is the unfolding operator associated to the $l(x)$-partition $\{x_k^\eps\}$, we have 
$$ \mathcal{T_\eps}(\chi^\eps) \longrightarrow \chi  \quad \hbox{s}-L^p\big( (0,1)\times Y^*\big), \hbox{ for } 1\leq p < \infty. $$
\end{proposition}

\begin{proof}
Since $T_\eps(\chi^\eps)$ and $\chi$ are uniformly bounded functions (they are actually bounded by the constant 1), it is enough to prove the convergence for $p=1$. Considering the set represented by each characteristic function we can write:
\begin{eqnarray*}
 \lefteqn{  \| \mathcal{T_\eps} (\chi^\eps)-  \chi \|_{L^1\big( (0,1)\times Y^*\big)} =}\\
& &= \int_0^1\int_{Y*} | \chi^\eps\Big( [x]_\eps +\Gamma_{[x]_\eps} y_1, \eps y_2\Big) \chi_{(0,l([x]_\eps))}(y_1) - \chi(x, y_1, y_2)| dy_1 dy_2 dx\\
& &= \int_0^1\int_0^{l([x]_\eps)} 
\int_0^{G\Big(  [x]_\eps +\Gamma_{[x]_\eps} y_1, \frac{1}{\eps}\big( [x]_\eps +\Gamma_{[x]_\eps} y_1\big)\Big)}  |1 -\chi(x, y_1, y_2)| dy_2 dy_1 dx\\
& &+
\int_0^1\int_0^{l([x]_\eps)} 
\int_{G\Big(  [x]_\eps +\Gamma_{[x]_\eps} y_1, \frac{1}{\eps}\big( [x]_\eps +\Gamma_{[x]_\eps} y_1\big)\Big)}^{G_1} | \chi(x, y_1, y_2)|dy_2 dy_1 dx \\
& & +\int_0^1\int_{l([x]_\eps)}^{l_1}\int_0^{G_1} | \chi(x, y_1, y_2)|\; dy_2 dy_1 dx.
\end{eqnarray*}

With the convergence 
\begin{equation}\label{convergence l}
l([x]_\eps) - l(x)\stackrel{\eps\to 0}\longrightarrow 0, \hbox{ a.e } x \in (0,1), 
\end{equation}
which follows easily from the convergence \eqref{1} and the continuity of the function $l(\cdot)$, together with 
\eqref{convergence 2} and with the aid of  the Lebesgue's Dominated Convergence Theorem, we easily prove that $\|\mathcal{T_\eps} (\chi^\eps)-  \chi \|_{L^1\big( (0,1)\times Y^*\big)}\to 0.$

\end{proof}

\begin{remark}\label{domain1}
The proposition above is basic to get the homogenized limit problem. By using this proposition  we will be able to pass to the limit  although our unfolded functions are not defined on a fixed domain. 
Then, it justifies to introduce the set $W$ which is the expected limit domain.
\end{remark}

\begin{remark}
Note that the choice of the points of the $l(x)$-partition is a key point to obtain the convergence of the domains. Somehow, the fact that the partition suitably reflects the geometry of the oscillating domain is very critical to obtain the convergence result. For example, let's see that things may go wrong if we do not choose wisely the admissible partition, even in the purely periodic case. If we assume that the oscillatory boundary of the thin domain is given by a function
$g: \R \to \R$ $L$-periodic and we consider the partition $$x_{0}^\eps=0 < x_{1}^\eps= \eps L_1 < x_{2 }^\eps=2\eps L_1< \cdots < x_{N^\eps}^\eps=N^\eps L_1 < x_{N^\eps+1}^\eps=1,$$ 
where $L_1$ is  rationally independent of $L$. Notice that this partition is an admissible one, but in this case, Proposition \ref{Convergence-of-G} does not hold, precisely because $L_1$ and $L$ are rationally independent. 
\end{remark}

\begin{proposition}\label{testconvergence}
Let $ \varphi \in L^p(0, 1)$, $1\leq p < \infty$. Then if $\mathcal{T}_\eps$ is the unfolding operator associated to the $l(x)$-partition, we have
$$ \mathcal{T_\eps(\varphi)}  \longrightarrow \varphi \chi \quad \hbox{s}-L^p\big( (0,1)\times Y^*\big).$$
\end{proposition}
\begin{proof}
The result is obvious for any $\varphi \in \mathcal{D}(0, 1)$. By density, if $\varphi \in  L^p(0, 1)$, let $\varphi_k \in \mathcal{D}(0, 1)$ such that $\varphi_k \rightarrow \varphi$ in $L^p(0, 1)$. Then, we have
\begin{eqnarray*}
\| \mathcal{T_\eps(\varphi)}  - \varphi \chi \|_{L^p\big( (0,1)\times Y^*\big)} \leqslant \| \mathcal{T_\eps(\varphi)}  -   \mathcal{T_\eps}(\varphi_k)  \|_{L^p\big( (0,1)\times Y^*\big)}\\
+ \| \mathcal{T_\eps}(\varphi_k)  - \varphi_k \chi \|_{L^p\big( (0,1)\times Y^*\big)} + \|\varphi_k \chi   - \varphi \chi \|_{L^p\big( (0,1)\times Y^*\big)}
\end{eqnarray*}
from which the result is straightforward.
\end{proof}

As we mentioned at the end of Section 2, see Remark \ref{important}, we need to establish a link between test functions in $R^\eps$ and test functions in $W$.  
Now, we study the convergence of these functions.

\begin{proposition}\label{convergence test}
Let $\psi=\psi(x,y_1,y_2)$ be a $C^\infty_\#(W)$ function.
We define the following function $\psi^{\eps} \in C^{\infty}(\overline {R^\eps})$ by
$$\psi^\eps(x, y) = \psi\big(x, \frac{x}{\eps}, \frac{y}{\eps}) \quad \forall (x, y) \in R^\eps.$$
Then if $\mathcal{T}_\eps$ is the unfolding operator associated to the $l(x)$-partition
$$ \mathcal{T_\eps(\psi^\eps)}  \longrightarrow \psi \chi \quad \hbox{s}-L^p\big( (0,1)\times Y^*\big), \forall 1\leq p < \infty.$$
\end{proposition}

\begin{proof}
We have already shown that $\psi^{\eps}$ is well defined, see Proposition \ref{properties}, v). Then, we only have to prove the convergence.
\begin{eqnarray*}
\lefteqn{\| \mathcal{T_\eps(\psi^\eps)}  - \psi \chi \|^p_{L^p\big( (0,1)\times Y^*\big)}=} \\
&&=  \int_0^1\int_{Y*} \big|\psi^\eps\Big( [x]_\eps +\Gamma_{[x]_\eps} y_1, \eps y_2\Big)\mathcal{T_\eps(\chi^\eps)}( x, y_1, y_2) 
 - \psi(x, y_1, y_2)\chi(x, y_1, y_2)\big|^p dy_1 dy_2 dx\\
&&= \int_0^1\int_{Y*} \big| \psi\Big( [x]_\eps +\Gamma_{[x]_\eps} y_1,\frac{1}{\eps} \big([x]_\eps +\Gamma_{[x]_\eps} y_1\big),  y_2\Big)\mathcal{T_\eps(\chi^\eps)}
 - \psi\chi\big|^p dy_1 dy_2 dx.
\end{eqnarray*}
Since $\psi$ is continuous on $W$ and taking into account the convergences  (\ref{1variable}) and (\ref{2variable}) we have 
\begin{equation}\label{uniform}
\big| \psi\Big( [x]_\eps  +\Gamma_{[x]_\eps} y_1,\frac{1}{\eps} \big([x]_\eps - l\Big([x]_\eps +\Gamma_{[x]_\eps} y_1\Big)\frac{[x]_\eps}{l([x]_\eps)} +\Gamma_{[x]_\eps} y_1\big), y_2\Big)- \psi(x, y_1, y_2 )\big|  \stackrel{\eps\to 0}\longrightarrow 0 \; \hbox{ a.e on } W.
\end{equation}
Hence, due to (\ref{uniform}), Proposition \ref{domainunfolding} and by applying the Lebesgue's Dominated Convergence Theorem we obtain the result.
\end{proof}

So far we have only considered functions in $L^p(R^\eps)$. Now we study the behavior of sequences in $W^{1,p}(R^\eps)$. We will prove
a compactness result which allows us to obtain the limit of the unfolded derivatives. Before that, we need to state and prove a technical Lemma which we will use later on. 

\begin{lemma}\label{technical}
Assume  $1\leq p < \infty$. For any function $\theta (\cdot) \in W^{1,p}_0(0, 1)$ there exists a function  $\psi$ in $L^p\Big((0,1); W^{1,p}_{\#}\big(Y^*(x)\big)\Big)$ such that
$$\psi=\psi(x,y_2), \hbox { in } W$$
$$\psi(x, y_2) = 0 \quad \hbox{on} \quad \partial{W} \setminus B_1,$$
$$\frac{1}{l(x)} \int_{Y^*(x)} \psi(x, y_2) dy_1dy_2 = \theta (x) \quad \forall x \in (0,1),$$
$$\| \psi \|_{L^p(W)} \leqslant C \|\theta\|_{L^p(0,1)},$$
where $B_1$ is the following lateral boundary of $W$
$$B_1 = \left\{ (x, 0, y_2) : x\in (0,1)\; \hbox{and} \; 0 < y_2 < G(x, 0)\right\} \bigcup \left\{ (x, l(x), y_2) : x\in (0,1)\; \hbox{and} \; 0 < y_2 < G(x, 0)\right\}.$$
\end{lemma}

\begin{proof}
Let us consider the cell $Y^*_0 = (0, l_0) \times (0, G_0) \subset Y^*(x), \; \forall x \in (0,1)$. We define the following auxiliary problem:
\begin{equation*} 
\left\{
\begin{gathered}
 - \frac{\partial^2 v}{{\partial y_2}^2}  = 1
\quad \textrm{ in } Y^*_0 \\
\frac{\partial v}{\partial y_1}  = 0
\quad \textrm{ in } Y^*_0\\
v = 0 \quad \textrm{on } A_1\cup A_2
\end{gathered}
\right.
\end{equation*}
where $A_1$ is the upper boundary and $A_2$ is the lower boundary of $Y^*_0.$

This problem admits an unique, nonzero solution that we can obtain explicitly:
$$ v(y_2) = - \frac{y_2^2}{2} + \frac{G_0}{2}y_2 \quad \forall y_2 \in (0, G_0).$$

Then we define the function $\psi$ by:
$$\psi(x, y_1, y_2) = \frac{l_0}{\int_{Y^*_0 }\Big( \frac{\partial v}{\partial y_2}\Big)^2 dy_1dy_2}\; \widetilde{v} (y_2) \theta(x) \quad \forall (x, y_1, y_2) \in W.$$
Since $\int_{Y^*_0 }\Big( \frac{\partial v}{\partial y_2}\Big)^2 dy_1 dy_2 = \int_{Y^*_0 } v\, dy_1 dy_2$,  it is easy to see  that $\psi$ satisfies all the properties of the Lemma.
\end{proof}

We are now in position to state the main result of this section.

\begin{theorem}\label{convergence prop}
Let $\varphi^\eps\in W^{1,p}(R^\eps)$ for $1\leq p<\infty$, with $|||\varphi^\eps|||_{W^{1,p}( R^\eps)} =  \eps^{-1/p} \|\varphi^\eps\|_{W^{1,p}(R^\eps)}$ uniformly bounded. Then, if $\mathcal{T}_\eps$ is the unfolding operator associated to the $l(x)$-partition
\begin{enumerate}
\item[i)] There exists a function $\varphi$ in $W^{1,p}(0,1)$ such that,  up to subsequences:
$$\mathcal{T_\eps(\varphi^\eps)}\rightharpoonup  \varphi \chi \quad \hbox{w}-L^p\big( (0,1)\times Y^*\big).$$
\item[ii)] There exists a function $\varphi_1$ in $L^p\Big((0,1); W^{1,p}_{\#}(Y^*(x))\Big)$ such that, up to subsequences:

\begin{eqnarray*}
\mathcal{T_\eps}\Big(\frac{\partial\varphi^\eps}{\partial x}\Big)\rightharpoonup \xi_0 (x, y_1, y_2)= 
 \left\{ 
\begin{array}{ll}
 \frac{\partial\varphi}{\partial x}(x) + l(x)\frac{\partial\varphi_1}{\partial y_1}(x, y_1, y_2) \quad \hbox{for} \quad (x, y_1, y_2) \in W \\
0 \quad \quad \hbox{for} \quad  (x, y_1, y_2) \in (0, 1) \times Y^* \setminus W.
\end{array}
\right.
\end{eqnarray*}
\begin{eqnarray*}
\mathcal{T_\eps}\Big(\frac{\partial\varphi^\eps}{\partial y}\Big)\rightharpoonup \xi_1 (x, y_1, y_2)= 
 \left\{ 
\begin{array}{ll}
l(x) \frac{\partial \varphi_1}{\partial y_2} (x, y_1, y_2) \quad \hbox{for} \quad (x, y_1, y_2) \in W \qquad \qquad\\
0 \quad \quad \hbox{for} \quad  (x, y_1, y_2) \in (0, 1) \times Y^* \setminus  W.
\end{array}
\right.
\end{eqnarray*}
\end{enumerate}
\end{theorem}

\begin{proof}
i)  Since the norm $|||\varphi^\eps|||_{L^p(R^\eps)}$ is uniformly bounded in $\eps$,  using Proposition \ref{properties}, iv), we deduce that $\|\mathcal{T_\eps(\varphi)}\|_{L^p\big( (0,1)\times Y^*\big)}$ is also uniformly bounded and, therefore,  we can extract a subsequence of $\mathcal{T_\eps(\varphi^\eps)}$, still denoted by $\mathcal{T_\eps(\varphi^\eps)}$ such that
$$\mathcal{T_\eps(\varphi^\eps)}\rightharpoonup  \hat{\varphi} \quad \hbox{w}-L^p\big( (0,1)\times Y^*\big).$$
To check that $ \hat{\varphi}$ is zero outside $W$ is obvious from Definition \ref{unfolding} and Proposition \ref{domainunfolding}. 

\par\medskip  Now we prove that $\hat{\varphi}$ does not depend on $(y_1, y_2)$ in $W$. For this, let $\Psi = (\psi_1, \psi_2)$ be a function in $[\mathcal{D}(W)]^2$. We also denote by $\Psi$ the extension by zero and notice that it belongs to $[\mathcal{D}\big((0, 1)\times Y^*\big)]^2.$ Using Proposition \ref{properties}, v),  we can define $\Psi^\eps\equiv  (\psi{_1}^\eps, \psi{_2}^\eps) \in [\mathcal{D}(R^\eps)]^2$, where $$\psi{_i}^\eps(x, y) = \psi_i(x, \frac{x}{\eps}, \frac{y}{\eps}) \;\; \forall \,(x, y) \in R^ \eps, \quad i=1, 2.$$
In addition, set

$$\theta{_1}^\eps(x,y) =\frac{\partial\psi_1}{\partial x}(x, \frac{x}{\eps}, \frac{y}{\eps}), \; \theta{_2}^\eps(x,y) =\frac{\partial\psi_1}{\partial y_1}(x, \frac{x}{\eps}, \frac{y}{\eps}) \hbox{ and } \theta{_3}^\eps(x,y) =\frac{\partial\psi_2}{\partial y_2}(x, \frac{x}{\eps}, \frac{y}{\eps}).$$
Integrating by parts, we have
$$\int_{R^\eps} \nabla\varphi^\eps(x, y)\cdot \Psi (x, \frac{x}{\eps},\frac{ y}{\eps}) dx dy = - \int_{R^\eps} \varphi^\eps (x, y) \Big(\theta{_1}^\eps(x,y) + \frac{1}{\eps}\theta{_2}^\eps(x,y) + \frac{1}{\eps}\theta{_3}^\eps(x,y)  \Big)dx dy.$$
Then, by the criterion for integrals, Proposition \ref{properties}, we obtain
\begin{eqnarray*}
 \lefteqn{ \int_{(0, 1) \times Y^*} \frac{\eps}{l([x]_\eps)} \left\{\mathcal{T_\eps}\Big(\frac{\partial\varphi^\eps}{\partial x}\Big)\mathcal{T_\eps} (\psi^\eps_1) + \mathcal{T_\eps}\Big(\frac{\partial\varphi^\eps}{\partial y}\Big)\mathcal{T_\eps}(\psi^\eps_2)
 \right\} dx dy_1 dy_2 }\\
&=  &-\int_{(0, 1) \times Y^*} \frac{1}{l([x]_\eps)} \left\{  \eps \mathcal{T_\eps(\varphi^\eps)}
\mathcal{T_\eps}(\theta{_1}^\eps)+ \mathcal{T_\eps(\varphi^\eps)}
\mathcal{T_\eps}(\theta{_2}^\eps) +\mathcal{T_\eps(\varphi^\eps)}
\mathcal{T_\eps}(\theta{_3}^\eps)\right\} dx dy_1 dy_2.
\end{eqnarray*}
Passing to the limit in both terms with the help of Proposition  \ref{domainunfolding} and Proposition \ref{convergence test}  we get
$$0 =  - \int_{W}  \frac{1}{l(x)}\hat{\varphi}(x, y_1, y_2) \hbox{div}_{y_1y_2}\Psi(x, y_1, y_2) dx dy_1 dy_2 \quad \forall \Psi \in [\mathcal{D} (W)]^2.$$
This implies that $\hat{\varphi}$ does not depend on $(y_1, y_2)$ in $W$. Then we can conclude that there exists a function $\varphi \in L^p(0, 1)$ such that:
$$\hat{\varphi} (x, y_1, y_2) = \varphi (x) \chi (x, y_1, y_2) \quad \forall (x, y_1, y_2) \in (0, 1) \times Y^*.$$
\par\medskip
We see now that $\varphi \in W^{1,p}(0,1)$. For this, for any function $\theta(x) \in \mathcal{D}(0, 1)$ let $\psi$ be the function in $W^{1,q}(W)$, $\frac{1}{p}+\frac{1}{q} = 1$, given by  Lemma \ref{technical}, that is
$$\frac{\partial \psi}{\partial y_1}  = 0
\quad \textrm{ in } W$$
$$\psi(x, y_2) = 0 \quad \hbox{on} \quad \partial{W} - B_1,$$
$$\frac{1}{l(x)} \int_{Y^*(x)} \psi(x, y_2) dy_1dy_2 = \theta (x) \quad \forall x \in (0,1),$$
$$\| \psi \|_{L^q(W)} \leqslant C \|\theta\|_{L^q(0,1)},$$
where $B1$ is the lateral boundary of W defined in Lemma \ref{technical}.
Note that the extension by zero  of $\psi$ in the direction $y_2$  belongs to the space $W^{1,q}((0,1)\times Y^*\big)$. We define the sequence $\{\psi^\eps\}$ by
$$\psi^\eps(x, y) = \psi(x, \frac{y}{\eps}) \quad \forall (x, y) \in R^\eps.$$
Integrating by parts, we obtain
$$\int_{R^\eps} \frac{\partial\varphi^\eps}{\partial x} (x, y) \psi^\eps (x, y) dx dy = - \int_{R^\eps} \varphi^\eps (x, y) \frac{\partial\psi^\eps}{\partial x}(x, y)dx dy.$$
Then, by the unfolding criterion for integrals, we have
\begin{equation}\label{previous}
\int_{(0, 1) \times Y^*} \frac{1}{l([x]_\eps)}\mathcal{T_\eps}\Big(\frac{\partial\varphi^\eps}{\partial x}\Big) \mathcal{T_\eps(\psi^\eps)} dx dy_1 dy_2
= -\int_{(0, 1) \times Y^*} \frac{1}{l([x]_\eps)} \mathcal{T_\eps(\varphi^\eps)}
\mathcal{T_\eps}\Big(\frac{\partial\psi^\eps}{\partial x}\Big) dx dy_1 dy_2.
\end{equation}
 We can pass to the limit in (\ref{previous}) using Proposition \ref{domainunfolding} , Proposition \ref{convergence test} and assuming that $\mathcal{T_\eps}\Big(\frac{\partial\varphi^\eps}{\partial x}\Big)\rightharpoonup \xi_0 (x, y_1, y_2)$ w-$L^p\big( (0,1) \times Y^*\big)$, it will be seen in the second part of Theorem. Thus, we get
\begin{equation}\label{derivative}
\int_{W} \frac{1}{l(x)}\xi_0  (x, y_1, y_2) \psi(x, y_2) dx dy_1 dy_2 = - \int_{W} \frac{1}{l(x)}\varphi(x) \frac{\partial\psi}{\partial x} (x, y_2) dx dy_1 dy_2.
\end{equation}
By using Lemma \ref{technical} above, the right-hand side of (\ref{derivative}) becomes $\int_{(0,1)} \varphi(x)\frac{\partial\theta}{\partial x} (x) dx,$ while the left-hand side is a linear continuous form in
$\theta(x) \in \mathcal{D}(0, 1).$ This implies that $\varphi \in W^{1,p}(0, 1).$\\

\par\medskip 
ii)  Since $|||\varphi^\eps|||_{W^{1,p}( R^\eps)} =  \eps^{-1/p}\|\varphi^\eps\|_{W^{1,p}(R^\eps)}$  is uniformly bounded, using property iv) in Proposition \ref{properties} we can extract two subsequences of $\mathcal{T_\eps}\Big(\frac{\partial\varphi^\eps}{\partial x}\Big)$ and $\mathcal{T_\eps}\Big(\frac{\partial\varphi^\eps}{\partial y}\Big)$, still denoted by  
$\mathcal{T_\eps}\Big(\frac{\partial\varphi^\eps}{\partial x}\Big)$ and $\mathcal{T_\eps}\Big(\frac{\partial\varphi^\eps}{\partial y}\Big)$ such that
$$\mathcal{T_\eps}\Big(\frac{\partial\varphi^\eps}{\partial x}\Big)\rightharpoonup \xi_0 (x, y_1, y_2) \quad \hbox{w}-L^p\big((0, 1) \times Y^*\big), $$
$$\mathcal{T_\eps}\Big(\frac{\partial\varphi^\eps}{\partial y}\Big)\rightharpoonup \xi_1 (x, y_1, y_2)  \quad \hbox{w}-L^p\big((0, 1) \times Y^*\big),$$
for certain $\xi_0$ and $\xi_1$ in $L^p\big((0, 1) \times Y^*\big).$

First note that $\xi_0 = \xi_1 = 0$ in $(0, 1) \times Y^* \backslash W$ follows from the definition of unfolding operator and Proposition \ref{domainunfolding}.\\

In order to find the precise form of $\xi_0$ and $\xi_1$ in $W$ we argue as follows. \\

Let $\Psi\equiv  (\psi_1, \psi_2) \in \big[C^\infty_\#(W)\big]^2$ be a function satisfying
$$\hbox{div}_{y_1y_2}\Psi = 0$$  
$$\Psi(x, y_1, y_2) \cdot n_{(y_1,y_2)}(x) = 0\; \hbox{on}\; \partial_{inf}Y^*(x) \cup \partial_{sup}Y^*(x),$$ 
$$\psi_1(0, \cdot, \cdot) = \psi_1(1, \cdot, \cdot)=0,$$
where $n_{(y_1,y_2)}(x)$ is the outward normal to $\partial Y^*(x)$ for every $x \in (0,1)$.

 From Proposition \ref{convergence test} we can define  
$\Psi^\eps\equiv  (\psi{_1}^\eps, \psi{_2}^\eps) \in [W^{1,p}(R^ \eps)]^2$, where $$\psi{_i}^\eps(x, y) = \psi_i(x, \frac{x}{\eps}, \frac{y}{\eps}) \;\; \forall \,(x, y) \in R^ \eps, \quad i=1, 2.$$
Then, integrating by parts, we have
\begin{align} \label{test}
&\int_{R^\eps} \Big[\nabla\varphi^\eps(x, y)- \nabla\varphi(x)\Big] \cdot \Psi^\eps(x, y) \;dx dy-\int_{\partial_{sup} R^\eps} \nu^\epsilon \cdot \Psi^\eps (\varphi^\eps - \varphi) \, dS= \nonumber\\
&=- \int_{R^\eps} \Big[\varphi^\eps(x, y) - \varphi(x)\Big] \hbox{div}_ {(x, y)}\Psi^\eps(x, y)\;dx dy \nonumber\\
&=  - \int_{R^\eps} \left\{\Big[\varphi^\eps (x, y) - \varphi(x)\Big] \frac{\partial\psi_1}{\partial x}(x, \frac{x}{\eps}, \frac{y}{\eps}) + \frac{1}{\eps} \Big[\varphi^\eps(x, y)- \varphi(x)\Big] \hbox{div}_{(y_1, y_2)}\Psi (x, \frac{x}{\eps}, \frac{y}{\eps}) \right\}dx dy \nonumber\\
&= - \int_{R^\eps} \Big[\varphi^\eps (x, y) - \varphi(x)\Big] \frac{\partial\psi_1}{\partial x}(x, \frac{x}{\eps}, \frac{y}{\eps})dx dy.
\end{align}
Applying the unfolding criterion for integrals we can write (\ref{test}) as:
\begin{align}\label{testunfold}
&\int_{(0, 1) \times Y^*} \frac{1}{l([x]_\eps)}\left\{\Big[\mathcal{T_\eps}\Big(\frac{\partial\varphi^\eps}{\partial x}\Big) - \mathcal{T_\eps}\Big(\frac{\partial\varphi}{\partial x}\Big)\Big] \mathcal{T_\eps}(\psi_1) +  \mathcal{T_\eps}\Big(\frac{\partial\varphi^\eps}{\partial y}\Big) \mathcal{T_\eps}(\psi_2)\right\}
dx dy_1 dy_2\nonumber\\
&=  -\int_{(0, 1) \times Y^*}  \frac{1}{l([x]_\eps)}\Big[\mathcal{T_\eps(\varphi^\eps)}- \mathcal{T_\eps(\varphi)}\Big]
\mathcal{T_\eps}\Big(\frac{\partial\psi_1}{\partial x}\Big) dx dy_1 dy_2+\frac{1}{\eps}\int_{\partial_{sup} R^\eps} \nu^\epsilon \cdot \Psi^\eps (\varphi^\eps - \varphi) \, dS.
\end{align}

Now, we prove that the boundary term vanishes as $\eps$ tends to zero. The unit normal to $\partial_{sup} R^\eps$ is given by
$$\nu^\eps = \Big(\frac{-\eps G_x - G_y}{\sqrt{(\eps G_x + G_y)^2+1}}, \frac{1}{\sqrt{(\eps G_x + G_y)^2+1}}\Big),$$
where $\displaystyle{G_x= \frac{\partial G}{\partial x}}$ and $G_y =\displaystyle{ \frac{\partial G}{\partial y}}$.

Then, taking into account $\Psi(x, y_1, y_2) \cdot n_{(y_1,y_2)}(x) = 0\; \hbox{on}\;\partial_{sup}Y^*(x)$ we have 
$$\int_{\partial_{sup} R^\eps} \nu^\epsilon \cdot \Psi^\eps (\varphi^\eps - \varphi) \, dS =-\eps \int_{\partial_{sup} R^\eps} \frac{G_x}{\sqrt{(\eps G_x + G_y)^2+1}}\psi_1^\eps (\varphi^\eps - \varphi)\, dS $$
Therefore, since $\displaystyle{\Big|\Big|\Big|\frac{G_x}{\sqrt{(\eps G_x + G_y)^2+1}}\psi_1^\eps (\varphi^\eps - \varphi)\Big|\Big|\Big|_{W^{1,1}( R^\eps)}}$ is uniformly bounded we can conclude
\begin{equation}\label{boundary}
\lim_{\eps\to0}\frac{1}{\eps} \int_{\partial_{sup} R^\eps} \nu^\epsilon\Psi^\eps (\varphi^\eps - \varphi) \, dS=0.
\end{equation}
Passing to the limit in \eqref{testunfold} with the help of Proposition \ref{domainunfolding}, Proposition \ref{convergence test}, the convergence in i) and \eqref{boundary} we get
\begin{equation*} 
\int_{W} \frac{1}{l(x)}\left \{ \Big[\xi_0(x, y_1, y_2) - \frac{\partial\varphi}{\partial x}(x) \Big] \psi_1(x, y_1, y_2) +  \xi_1(x, y_1, y_2) \psi_2(x, y_1, y_2) \right\} dx dy_1 dy_2 = 0.
 \end{equation*}
 Hence, we obtain
 \begin{equation}
\int_{W}  \Big( \frac{1}{l(x)}\Big[\xi_0(x, y_1, y_2) - \frac{\partial\varphi}{\partial x}(x) \Big], \frac{1}{l(x)}\xi_1(x, y_1, y_2)\Big) \cdot \big( \psi_1, \psi_2\big)(x, y_1, y_2)
dx dy_1 dy_2 = 0.
 \end{equation}

The Helmholtz decomposition, see \cite{DJL}, yields that the orthogonal of divergence-free functions is exactly the gradients. Then, we can conclude that there exists a function
$\varphi_1 \in L^p\Big((0,1); W^{1,p}_{\#}\big(Y^*(x)\big)\Big)$ such that
$$\frac{1}{l(x)}\Big(\xi_0(x, y_1, y_2) -   \frac{\partial\varphi}{\partial x}(x)\Big) = \frac{\partial\varphi_1}{\partial y_1}(x, y_1, y_2) \quad \forall (x, y_1, y_2) \in W,$$
$$\frac{1}{l(x)}\xi_1(x, y_1, y_2) = \frac{\partial\varphi_1}{\partial y_2}(x, y_1, y_2)\quad \forall (x, y_1, y_2) \in W,$$
which ends the proof of the Theorem.

\end{proof}

\begin{remark} \label{compactness theorem}
As we wrote in the introduction, this Theorem is the main tool to obtain the homogenized limit problem. On one hand, it shows that the limit $\varphi$ lies in $W^{1,p}(0,1),$ which was not clear at all in view
of the a priori estimates because $\mathcal{T_\eps(\varphi^\eps)}$ is defined on a varying set. On the other hand, it allows us to relate the limit of the unfolded derivatives $\mathcal{T_\eps}\Big(\frac{\partial\varphi^\eps}{\partial x}\Big)$ and $\mathcal{T_\eps}\Big(\frac{\partial\varphi^\eps}{\partial y}\Big)$  with the weak derivative of $\varphi$. Note that the variable period plays a decisive role in the limit and,
as we will see in the next section, it enters into the limit equation.
\end{remark}

\begin{remark}
The proof provided for Theorem \ref{convergence prop} uses similar techniques as in \cite{AL}. It is also possible to obtain the same result by adapting the scale-splitting operators, $\mathcal{Q_\eps}$ and $\mathcal{R_\eps}$, introduced in \cite{CDG} to the new situation presented in this paper. However, the calculations using these operators are a little bit more involved. 
\end{remark}

\section{Homogenization of the Neumman problem}\label{Sec-Neumann}

In this section we return to the problem (\ref{OP0}) presented in the Introduction  and we show how the unfolding operator method adapted to this new situation allows to obtain the homogenized limit problem.  We will need the results from the previous Sections and in particular the convergence result from Theorem \ref{convergence prop}. Therefore, throughout this Section we will assume that the unfolding operator $\mathcal{T}_\eps$ is the one associated to the $l(x)$-partition. 

We start by recalling the problem:

\begin{equation} \label{OP1}
\left\{
\begin{gathered}
- \Delta u^\epsilon + u^\epsilon = f^\epsilon
\quad \textrm{ in } R^\epsilon \\
\frac{\partial u^\epsilon}{\partial \nu ^\epsilon} = 0
\quad \textrm{ on } \partial R^\epsilon
\end{gathered}
\right. 
\end{equation}
with $f^\epsilon \in L^2(R^\epsilon)$ and where $\nu^\epsilon = (\nu^\epsilon_1, \nu^\epsilon_2)$ is the unit outward normal to $\partial R^\epsilon$
and $\frac{\partial }{\partial \nu^\epsilon}$ is the outside normal derivative. The domain $R^\eps$ is a two dimensional thin domain which is given by
 $$R^\epsilon = \{ (x,y) \in \R^2 \; | \;  x \in (0,1), \, 0 < y < \epsilon \, G(x, x/\eps) \}$$
 where $G(\cdot, \cdot)$ is a function satisfying the assumptions established in Section 2.
From Lax-Milgran Theorem, we have that  problem (\ref{OP1}) has a unique solution for each $\epsilon > 0$. 
We are interested here in analyzing the behavior of the solutions as $\eps\to 0$.

The variational formulation of (\ref{OP1}) is
\begin{eqnarray} \label{WF}
\left\{
\begin{array}{lll}
\text{Find}\; u^\epsilon \in H^1(R^\epsilon) \; \text{such that}\\ \\
\displaystyle \int_{R^\epsilon} \Big\{ \nabla u^\eps \cdot \nabla \varphi
+ u^\epsilon \varphi \Big\} dx dy  =  \int_{R^\epsilon} f^\epsilon \varphi dx dy,\\ \\
\forall \varphi \in H^1(R^\epsilon). 
\end{array}
\right.
\end{eqnarray}

Now we are in condition to state and prove the homogenization result.
\begin{theorem}\label{limit problem}
Let $u^\eps$ be the solution of problem (\ref{OP1}). Assume that $f^\eps \in L^2(R^\eps)$ satisfies $||| f^\eps |||_{L^2(R^\epsilon)} \leq C$ with C independent of the parameter $\eps$ and, therefore, there  exists $\hat{f} \in L^2(W)$ such that, via subsequences, $\displaystyle{\mathcal{T_\eps}(f^\eps)\rightharpoonup  \hat{f} \chi \quad \hbox{weakly in }\; L^2\big( (0,1)\times Y^*\big)}$.
Then, there exist $u \in H^1(0, 1)$ and $u_1\in L^2\Big((0,1); H^1_{\#}\big(Y^*(x)\big)\Big)$ such that
\begin{eqnarray}
& &\displaystyle{\mathcal{T_\eps}(u^\eps)\rightharpoonup  \hat{u}  =  u \chi \quad \hbox{weakly in }\; L^2\big( (0,1)\times Y^*\big),}\label{solu}\\ 
& &\displaystyle{\mathcal{T_\eps}\Big(\frac{\partial u^\eps}{\partial x}\Big)\rightharpoonup \xi_0 (x, y_1, y_2)  = 
 \frac{\partial u}{\partial x}(x) +l(x)\frac{\partial u_1}{\partial y_1}(x, y_1, y_2) \quad  \hbox{weakly in }\; L^2(W),}\label{dev1} \\ 
 & &\displaystyle{\mathcal{T_\eps}\Big(\frac{\partial u^\eps}{\partial y}\Big)\rightharpoonup \xi_1 (x, y_1, y_2)
 =  l(x)\frac{\partial u_1}{\partial y_2}(x, y_1, y_2)\quad  \quad \hbox{weakly in }\; L^2(W),}\label{dev2}
\end{eqnarray}
and the pair $(u, u_1)$ is the unique solution of the problem

\begin{equation} \label{system homogenized}
\left\{
\begin{gathered}
 \forall \phi \in H^1(0, 1), \forall \psi \in L^2\Big((0,1); H^1_{\#}\big(Y^*(x)\big)\Big) \\
 \int_W \left\{ \Big( \frac{1}{l(x)}\frac{\partial u}{\partial x}(x) + \frac{\partial u_1}{\partial y_1}(x, y_1, y_2)\Big) \Big( \frac{\partial \phi}{\partial x}(x) + \frac{\partial \psi}{\partial y_1}(x, y_1, y_2)\Big)\right\} dxdy_1dy_2 \\
 + \int_W \left\{ \, \frac{\partial u_1}{\partial y_2}(x, y_1, y_2)\frac{\partial \psi}{\partial y_2}(x, y_1, y_2) + \frac{u(x) \phi(x)}{l(x)} \right\} dxdy_1dy_2 = \int_W \frac{\hat{f}(x,y_1,y_2) \phi(x)}{l(x)} dxdy_1dy_2. 
\end{gathered}
\right.
\end{equation}

Equivalently, $u \in H^1(0,1)$ is the unique weak solution of the following Neumann problem, it was obtained through the relation $ u_1 (x, y_1, y_2) = - X(x) (y_1, y_2) 
 \frac{1}{l(x)}\frac{\partial u}{\partial x}(x)$, 
\begin{equation} \label{homogenized problem}
\left\{
\begin{gathered}
- \big(r(x) u_x\big)_x + \frac{|Y^*(x)|}{l(x)}u = f_0, \quad x \in (0,1) \\
 u'(0) = u' (1) = 0
\end{gathered}
\right.
\end{equation}

where 

\begin{equation} \label{RPFL1}
\begin{gathered}
 f_0=  \frac{1}{l(x)}\int_{Y^*(x)} \hat{f}dy_1 dy_2,\\ 
\end{gathered}
\end{equation}
\begin{equation} \label{RPFL}
\begin{gathered}
r(x) =  \frac{1}{l(x)}\int_{Y^*(x)}\Big\{ 1 - \frac{\partial X(x)}{\partial y_1}(y_1,y_2) \Big\} dy_1 dy_2\\ 
\end{gathered}
\end{equation}
and $X(x)$ is the unique solution which is $l(x)$-periodic in the first variable, of the problem
\begin{equation} \label{AUXG}
\left\{
\begin{gathered}
- \Delta X(x)  =  0  \textrm{ in } Y^*(x)  \\
\frac{\partial X(x)}{\partial N}  =  0  \textrm{ on } B_2(x)  \\
\frac{\partial X(x)}{\partial N}  =  N_1(x) \textrm{ on } B_1(x)  \\
\int_{Y^*(x)} X(x) \; dy_1 dy_2  =  0  
\end{gathered}
\right.
\end{equation}
in the representative cell $Y^*(x)$ given by
\begin{equation} \label{CELLL1}
Y^*(x) = \{ (y_1,y_2) \in \R^2 \; : \; 0< y_1 < l(x), \quad 0 < y_2 < G(x,y_1) \}. 
\end{equation}
 $B_1(x)$ is the upper boundary and $B_2(x)$ is the lower
boundary of $\partial Y^*(x)$ for all $x \in I$.
\end{theorem}

\begin{remark}
Observe that the limit equation (\ref{homogenized problem}) reflects the geometry of the thin domain. As it is reasonable to expect , the family of solutions converge to a function of just one variable which satisfies a
elliptic equation in one dimension and the variable period appears explicitly in this equation. 
\end{remark}

\begin{proof}
We start by establishing  a priori estimates of $u^\eps$. In fact, taking $\varphi = u^\eps$ in the variational formulation (\ref{WF}), we obtain
\begin{equation} \label{priori1}
\begin{gathered}
\Big\| \frac{\partial u^\epsilon}{\partial x} \Big\|_{L^2(R^\epsilon)}^2
+ \Big\| \frac{\partial u^\epsilon}{\partial y} \Big\|_{L^2(R^\epsilon)}^2
+ \| u^\epsilon \|_{L^2(R^\epsilon)}^2
\le \| f^\eps \|_{L^2(R^\epsilon)} \| u^\epsilon \|_{L^2(R^\epsilon)}.
\end{gathered}
\end{equation}
Consequently, 
\begin{equation} \label{priori}
\begin{gathered}
\Big|\Big|\Big| \frac{\partial u^\epsilon}{\partial x} \Big|\Big|\Big|_{L^2(R^\epsilon)}^2
+ \Big|\Big|\Big| \frac{\partial u^\epsilon}{\partial y} \Big|\Big|\Big|_{L^2(R^\epsilon)}^2
+ ||| u^\epsilon |||_{L^2(R^\epsilon)}^2
\le ||| f^\eps |||_{L^2(R^\epsilon)} ||| u^\epsilon |||_{L^2(R^\epsilon)}.
\end{gathered}
\end{equation}
Taking into account that there exists $C >0$, independent of $\eps$, such that $ ||| f^\eps  |||_{L^2(R^\epsilon)} ||| \leqslant C$, we obtain
\begin{equation} \label{EST00}
\begin{gathered}
||| u^\epsilon |||_{L^2(R^\epsilon)}, \Big|\Big|\Big|\frac{\partial u^\epsilon}{\partial x} \Big|\Big|\Big|_{L^2(R^\epsilon)}
\textrm{ and }  \Big|\Big|\Big| \frac{\partial u^\epsilon}{\partial y}  \Big|\Big|\Big|_{L^2(R^\epsilon)} 
\le C \quad \forall \epsilon > 0.
\end{gathered}
\end{equation}

Therefore, the compactness Theorem \ref{convergence prop} implies that there exist $u \in H^1(0,1)$ and $u_1 \in L^2\Big((0,1); H^1_{\#}\big(Y^*(x)\big)\Big)$ such that
\begin{eqnarray}\label{main convergence}
\begin{array}{lll}
\displaystyle{\mathcal{T_\eps}(u^\eps)\rightharpoonup  \hat{u} = u \chi \quad \hbox{weakly in }\; L^2\big( (0,1)\times Y^*\big),}\\ \\
\displaystyle{\mathcal{T_\eps}\Big(\frac{\partial u^\eps}{\partial x}\Big)\rightharpoonup \xi_0 (x, y_1, y_2) = 
 \frac{\partial u}{\partial x}(x) +l(x)\frac{\partial u_1}{\partial y_1}(x, y_1, y_2) \quad  \hbox{weakly in }\; L^2(W),} \\ \\
\displaystyle{\mathcal{T_\eps}\Big(\frac{\partial u^\eps}{\partial y}\Big)\rightharpoonup \xi_1 (x, y_1, y_2)
= l(x)\frac{\partial u_1}{\partial y_2}(x, y_1, y_2)\quad  \quad \hbox{weakly in }\; L^2(W).}\\
\end{array}
\end{eqnarray}


We are now in the position of finding the homogenized equations satisfied by $u$ and $u_1$. Let us apply the unfolding operator to the original variational formulation (\ref{WF}). For $\phi \in H^1(0,1)$,
by the unfolding criterion for integrals (\ref{u.c.i}), we have
\begin{eqnarray*}
\int_{(0, 1) \times Y^*} \frac{1}{l([x]_\eps)}\left\{\mathcal{T_\eps}\Big(\frac{\partial u^\eps}{\partial x}\Big) \mathcal{T_\eps}(\frac{\partial \phi}{\partial x}) +  \mathcal{T_\eps}(u^\eps)\mathcal{T_\eps}(\phi) \right\}
dx dy_1 dy_2
=  \int_{(0, 1) \times Y^*}  \frac{1}{l([x]_\eps)}\mathcal{T_\eps}(f^\eps)  \mathcal{T_\eps(\phi)} dx dy_1 dy_2.
\end{eqnarray*}
Observe that in this last equality we have taken $\phi \in H^1 (0, 1)$ and the term including partial derivative with respect to $y$ does not appear.
By the convergences of $(\ref{main convergence})$ together with Proposition \ref{testconvergence} we can pass to the limit in the last equality and we obtain the first equation:
\begin{eqnarray} \label{first equation}
\int_W \left\{ \Big( \frac{1}{l(x)}\frac{\partial u}{\partial x}(x) + \frac{\partial u_1}{\partial y_1}(x, y_1, y_2)\Big)\frac{\partial \phi}{\partial x}(x)
+  \frac{u(x) \phi(x)}{l(x)} \right\} dxdy_1dy_2 =\nonumber\\
= \int_W \frac{\hat{f}(x) \phi(x)}{l(x)} dxdy_1dy_2, \quad \forall \phi \in H^1(0,1). 
\end{eqnarray}

We take now as a test function in (\ref{WF}) the function $v^\eps$ defined by:
$$v^\eps(x, y) = \eps \psi(x, \frac{x}{\eps}, \frac{y}{\eps} )\quad \forall (x, y) \in R^\eps$$ 
where $\psi \in C^\infty_\#(W)$.

It is obvious from the definition that $v^\eps \in H^1(R^\eps)$. Furthermore, it satisfies
$$ \frac{\partial v^\eps}{\partial x} = \eps \frac{\partial \psi}{\partial x} + \frac{\partial \psi}{\partial y_1},$$
$$ \frac{\partial v^\eps}{\partial y} = \frac{\partial \psi}{\partial y_2}.$$
Hence, using the properties of the unfolding operator we can show,
\begin{eqnarray}\label{limit1}
\begin{array}{lll}
\displaystyle{\mathcal{T_\eps}(v^\eps) \rightarrow 0 \quad \hbox{s-}L^2((0,1)\times Y^*),}\\ \\
\displaystyle{\mathcal{T_\eps}\Big(\frac{\partial v^\eps}{\partial x}\Big) \rightarrow  \frac{\partial \psi}{\partial y_1} \chi \quad \hbox{s-}L^2((0,1)\times Y^*),}\\ \\
\displaystyle{\mathcal{T_\eps}\Big(\frac{\partial v^\eps}{\partial y}\Big) \rightarrow  \frac{\partial \psi}{\partial y_2}\chi  \quad \hbox{s-}L^2((0,1)\times Y^*).}\\ \\
\end{array}
\end{eqnarray}
Due to the unfolding criterion for integrals, from the variational formulation (\ref{WF}) we obtain
\begin{eqnarray}\label{eq2}
\int_{(0, 1) \times Y^*} \frac{1}{l([x]_\eps)}\left\{\mathcal{T_\eps}\Big(\frac{\partial u^\eps}{\partial x}\Big) \mathcal{T_\eps}\Big(\frac{\partial v^\eps}{\partial x}\Big) + \mathcal{T_\eps}\Big(\frac{\partial u^\eps}{\partial y}\Big) \mathcal{T_\eps}\Big(\frac{\partial v^\eps}{\partial y}\Big)
+ \mathcal{T_\eps}(u^\eps)\mathcal{T_\eps}(v^\eps) \right\}
dx dy_1 dy_2 \nonumber \\
=  \int_{(0, 1) \times Y^*} \frac{1}{l([x]_\eps)}\mathcal{T_\eps}(f^\eps)  \mathcal{T_\eps}(v ^\eps) dx dy_1 dy_2.
\end{eqnarray}

Now using  the three statements in (\ref{limit1}) and (\ref{main convergence}), we pass to the limit in (\ref{eq2}) and we obtain the second equation:
\begin{eqnarray}\label{second equation}
&&\int_W  \Big( \frac{1}{l(x)}\frac{\partial u}{\partial x}(x) + \frac{\partial u_1}{\partial y_1}(x, y_1, y_2)\Big) \frac{\partial \psi}{\partial y_1}(x, y_1, y_2)\; dxdy_1dy_2 \nonumber \\
&& +\int_W \frac{\partial u_1}{\partial y_2}(x, y_1, y_2)\frac{\partial \psi}{\partial y_2}(x, y_1, y_2) \; dxdy_1dy_2= 0, \; \forall \psi \in C^\infty_\#(W).
\end{eqnarray}

By density, (\ref{second equation}) holds true for any function $\psi \in L^2\Big((0,1); H^1_{\#}\big(Y^*(x)\big)\Big)$. Therefore, by summing (\ref{first equation}) and (\ref{second equation}) we have the
homogenized system (\ref{system homogenized}).

To end the proof we will see the relation between the homogenized system  and the classical homogenized equation (\ref{homogenized problem}). This is achieved by using the solutions of the problem (\ref{AUXG}).
Treating $x$ as a parameter in ($\ref{second equation}$) is easy to check that it is a variational formulation associated to the following cell-problem:
\begin{equation} \label{AUXG1}
\left\{
\begin{gathered}
- \Delta u_1(x)  =  0  \textrm{ in } Y^*(x)  \\
\frac{\partial u_1(x)}{\partial N}  =  0  \textrm{ on } B_2(x)  \\
\frac{\partial  u_1(x)}{\partial N}  =  \frac{- N_1(x)}{l(x)}\frac{\partial u}{\partial x}(x)\textrm{ on } B_1(x)  \\
  u_1(x, \cdot, \cdot) \; l(x)-\hbox{periodic in the variable} \; y_1,\\
\end{gathered}
\right.
\end{equation}
where $N(x) = (N_1(x), N_2(x))$ is the unit outward normal to $\partial Y^*(x)$, $B_1(x)$ is the upper boundary and $B_2(x)$ is the lower
boundary of $\partial Y^*(x)$ for all $x \in I$.  
Thus, taking into account  $u$ is independent of $(y_1, y_2)$ one can see inmediatly that:
\begin{equation}\label{u_1}
u_1 (x, y_1, y_2) = - X(x) (y_1, y_2) 
\frac{1}{l(x)} \frac{\partial u}{\partial x}(x)\; (x, y_1, y_2) \in W,
\end{equation}
where $X(x)$ is the solution of (\ref{AUXG}).\\
Replacing $u_1$ by its value, (\ref{u_1}), in the equation (\ref{first equation}) we obtain the weak formulation of (\ref{homogenized problem}).

 The uniqueness and existence of weak solution of the problem (\ref{homogenized problem}) is an inmediate consequence of the Lax-Milgram theorem.

\end{proof}

\begin{remark}
If the non homogeneous term $f^\eps(x, y)$ is a fixed function depending only on the first variable, that is, $f^\eps(x, y)=f(x)$, it is easy to see that $f_0(x)=\frac{|Y^*(x)|}{l(x)}f(x)$ and therefore, \eqref{homogenized problem} can be written as  
$$
\left\{
\begin{gathered}
- \frac{l(x)}{|Y^*(x)|}\big(r(x) u_x\big)_x + u = f, \quad x \in (0,1) \\
 u'(0) = u' (1) = 0
\end{gathered}
\right.
$$
\end{remark}

\begin{remark}
Notice that in case $G_\eps$ presents a purely periodic behavior we recover the homogenized limit problem obtained in \cite{ArrCarPerSil}. On the other hand, in case the amplitude of the oscillation
depends on $x$ but the period is constant, $l(x) \equiv L$, we obtain the same homogenized limit problem as in \cite{ArrPer2011}.
\end{remark}

\section{Corrector result for the Neumann problem}\label{Sec-corrector}
In this section we address the question of correctors for problem (\ref{OP1}). To do that, we need the averaging operator $\mathcal{U}_\eps$, adapted to locally periodic thin domains, see \cite{CDG1} for the
definition for a purely periodic case.  In principle, this operator could be associated to  any ``admissible partition'' $\{x_k^\eps\}$ but since we will use it in connection with convergence properties of the solutions and will use the results from previous sections, we will consider it  is already associated to the $l(x)$-partition, see Definition \ref{partition}.  Hence, we define 
 \begin{definition}
 Let $\{x_k^\eps\}$ be the $l(x)$-partition. Then, if $\varphi \in L^p((0,1) \times Y^*)$, $ p \in [1, \infty]$, we set
 $$\mathcal{U_\eps}(\varphi) (x, y) = \frac{1}{l([x]_\eps)}\int_{0}^ {l([x]_\eps)}\varphi\Big([x]_\eps + \Gamma_{[x]_\eps} y_1, \frac{x - [x]_\eps}{\Gamma_{[x]_\eps}}, \frac{y}{\eps}\Big)\; dy_1, 
 \quad \forall (x, y) \in R^\eps.$$
 \end{definition}

 \begin{proposition} The main properties of $\mathcal{U}_\eps$ are the following:
 \begin{enumerate}
  \item[i)]  Assume $1\leq p, q \leq \infty$ and $\frac{1}{p} + \frac{1}{q} = 1$. The averaging operator $\mathcal{U}_\eps$ is the formal adjoint of the unfolding operator $\mathcal{T_\eps}$, in the sense that
  $$ \int_{(0, 1)\times Y^*} \frac{1}{l([x]_\eps)} \mathcal{T_\eps(\varphi)} \psi\;  dx dy_1 dy_2 = \frac{1}{\eps} \int_{R^\eps} \varphi \mathcal{U_\eps}(\psi) \; dxdy,\;  \forall \varphi \in L^q(R^\eps) \hbox{ and } \psi \in L^p((0,1) \times Y^*).$$
 \item[ii)] \label{u}Let $p$ belong to $[1, \infty]$. The averaging operator $\mathcal{U}_\eps$ is linear continuous from $L^p((0,1) \times Y^*)$ to $L^p(R^\eps)$ and there exists a constant $C>0$ independent of $\eps$ such that
 $$ |||\mathcal{U_\eps}(\varphi)|||_{L^p(R^\eps)} \leqslant C \|\varphi\|_{L^p\big( (0,1)\times Y^*\big)}, \quad \forall \varphi \in L^p((0,1) \times Y^*).$$ 
\item[iii)] $\mathcal{U_\eps}$ is the left inverse of $\mathcal{T_\eps}$, that is $(\mathcal{U_\eps} \circ \mathcal{T_\eps}) (\phi) = \phi$ for every $\phi \in L^p(R^\eps), 1\leq p \leq \infty. $
\item[iv)] Suppose that $p \in [1, \infty)$. Let $\varphi \in L^p(0, 1)$. Then, $|||\mathcal{U_\eps}(\varphi) - \varphi|||_{L^p(R^\eps)} \to 0$ when $\eps \to 0.$
\item[v)] Let $\{\varphi^\eps\}$ be a sequence in $L^p(R^\eps)$, $p \in [1, \infty)$, such that
$\mathcal{T_\eps(\varphi^\eps)}  \to \varphi  \hbox{ s}-L^p\big( (0,1)\times Y^*\big)$. Then 
$$|||\mathcal{U_\eps}(\varphi) - \varphi^\eps|||_{L^2(R^\eps)} \to 0.$$ 
  \end{enumerate}
  \end{proposition}
 \begin{proof}
 \begin{enumerate}
 \item[i)] For every $\varphi \in L^q(R^\eps) $ and $\psi \in L^p((0,1) \times Y^*)$ we have
 \begin{eqnarray*}
\lefteqn{ \int_{(0, 1)\times Y^*} \frac{1}{l([x]_\eps)} \mathcal{T_\eps(\varphi)} \psi\;  dx dy_1 dy_2 =}\\
&=& \sum_{k=0}^{N_\eps}\int_{(x_{k}^\eps, x_{k+1}^\eps) \times Y^*} \frac{1}{l(x_{k}^\eps)}\widetilde{\varphi} \Big(x_{k}^\eps +\Gamma_{x_{k}^\eps} y_1, \eps y_2\Big)\chi_{(0,l(x_{k}^\eps))}(y_1)  \psi \;dx dy_1 dy_2 \\
 &=& \sum_{k=0}^{N_\eps}\int_{(0, l(x_{k}^\eps)) \times Y^*} \frac{1}{l(x_{k}^\eps)}\widetilde{\varphi} \Big(x_{k}^\eps +\Gamma_{x_{k}^\eps} y_1, \eps y_2\Big)\chi_{(0,l(x_{k}^\eps))}(y_1) \psi\Big(x_{k}^\eps +\Gamma_{x_{k}^\eps}z, y_1, y_2\Big) \Gamma_{x_{k}^\eps} \; dz dy_1 dy_2\\
& = & \sum_{k=0}^{N_\eps}\int_{(0, l(x_{k}^\eps)) \times (x_{k}^\eps, x_{k+1}^\eps) \times (0, G_1)} \frac{1}{l(x_{k}^\eps)} \widetilde{\varphi} (x, \eps y_2) \psi\Big(x_{k}^\eps +\Gamma_{x_{k}^\eps}z, 
 \frac{x- x_{k}^\eps}{ \Gamma_{x_{k}^\eps}}, y_2\Big)\; dz dx dy_2\\
 &=& \sum_{k=0}^{N_\eps}\frac{1}{\eps}\int_{(x_{k}^\eps, x_{k+1}^\eps) \times (0,\eps G_1)}  \widetilde{\varphi} (x, y)\Big(\frac{1}{l(x_{k}^\eps)}\int_{(0, l(x_{k}^\eps))} \psi\Big(x_{k}^\eps +\Gamma_{x_{k}^\eps}z, 
 \frac{x- x_{k}^\eps}{ \Gamma_{x_{k}^\eps}}, \frac{y}{\eps}\Big)\; dz\Big)\; dx dy\\
 &=& \frac{1}{\eps} \int_{R^\eps} \varphi \mathcal{U_\eps}(\psi) \; dxdy.
\end{eqnarray*}
\item[ii)] It is a immediate consequence of the duality above and of the property iv) in Proposition \ref{properties}.
\item[iii)] Simple consequence of the definition of the operator $\mathcal{U_\eps}$.
\item[iv)] The result is clear for any $\varphi \in \mathcal{D}(0, 1)$. By density, we obtain the convergence.
\item[v)] It is a direct consequence of the properties ii) and iii). Observe that
$$|||\mathcal{U_\eps}(\varphi) - \varphi^\eps|||_{L^p(R^\eps)} = |||\mathcal{U_\eps}\big(\varphi - \mathcal{T_\eps(\varphi^\eps)}\big)|||_{L^p(R^\eps)} 
\leqslant C \|\varphi - \mathcal{T_\eps(\varphi^\eps)}\|_{L^p\big( (0,1)\times Y^*\big)}.$$
\end{enumerate}
  \end{proof}

Finally, we give a general corrector result. We show convergence in $H^1$-norms if we add the first-order corrector to the original solutions $u^\eps$.
 \begin{theorem}\label{correctors}
Assume hypotheses of Theorem \ref{limit problem}  hold. Then,  
\begin{enumerate}
\item[i)] $\displaystyle{\lim_{\eps \to 0}|||u^\eps - u|||_{L^2(R^\eps)}=0.}$
\medskip
\item[ii)] $\displaystyle{\lim_{\eps \to 0}|| \mathcal{T_\eps}(\nabla u^\eps) - \big( \nabla u  + l(x)(\nabla_{y_1y_2}u_1)\big)\chi||_{\Big[L^2\big( (0,1)\times Y^*\big)\Big]^2}=0.}$
\medskip 
\item[iii)] $\displaystyle{\lim_{\eps \to 0}|||\nabla u^\eps - \nabla u - l(x)\mathcal{U_\eps}(\nabla_{y_1y_2}u_1)|||_{[L^2(R^\eps)]^2}=0.}$
\medskip
\item[iv)] $\displaystyle{\lim_{\eps \to 0}||| u^\eps - u + \eps \frac{\partial u}{\partial x}X(x, x/\eps, y / \eps)|||_{H^1(R^\eps)}=0.}$
\end{enumerate}
 \end{theorem}
 
 \begin{remark}
 Notice that the special feature of the first-order corrector function, $\eps \frac{\partial u}{\partial x}X(x, x/\eps, y / \eps)$, is that its dependence of $x$ involves the spatial changes of the basic
 cell $Y^*(x)$ through the solution $X(x)$ to the auxiliary problem.
 \end{remark}
 
 \begin{proof}
 \begin{enumerate}
 \item[i)] Using the change of variables $(x, y) \to (x, y/\eps)$ and extension operators, it is easy to prove this convergence.
  \item[ii)]  This convergence improve the convergences (\ref{dev1}) and (\ref{dev2}). It is based on the convergence of the energy.
  Taking $\varphi = \frac{u^\eps}{l(x)}$ in the variational formulation (\ref{WF}), we obtain
\begin{eqnarray*}
\int_{R^\eps} \Big( \frac{1}{l(x)} \Big({\frac{\partial u^\eps}{\partial x}}\Big)^2 - u_\eps \frac{\partial u^\eps}{\partial x} \frac{l'(x)}{l(x)^2} + \frac{1}{l(x)} \Big({\frac{\partial u^\eps}{\partial y}}\Big)^2\Big) \; dxdy = 
\int_{R^\eps} \Big(\frac{f^\eps u^\eps}{l(x)} - \frac{ ({u^\eps}) ^2}{l(x)}\Big) \; dxdy.
\end{eqnarray*}
Therefore, using the unfolding criterion for integrals and passing to the limit we get the following convergence
\begin{eqnarray}\label{conver1}
\lefteqn{\int_{(0, 1) \times Y^*} \frac{1}{l(x)^2}\left\{{\mathcal{T_\eps}\Big(\frac{\partial u^\eps}{\partial x}\Big)}^2 + {\mathcal{T_\eps}\Big(\frac{\partial u^\eps}{\partial y}\Big)}^2\right\}
dx dy_1 dy_2} \nonumber \\
&\stackrel{\eps\to 0}\longrightarrow & \int_{W} \left\{ \frac{\hat{f}u - u^2}{l(x)^2} + \frac{1}{l(x)^2} \Big( u \frac{\partial u}{\partial x} \frac{l'(x)}{l(x)} +  u \frac{\partial u_1}{\partial y_1} l'(x)\Big) \right\}dx dy_1 dy_2.
\end{eqnarray}
On the other hand, choosing $\phi = \frac{u}{l(x)}$ and $\psi = u_1$ as test functions in (\ref{system homogenized}) we get
\begin{eqnarray}\label{conver2}
\lefteqn{\int_W  \left\{ \Big(\frac{1}{l(x)}{\frac{\partial u}{\partial x}} + \frac{\partial u_1}{\partial y_1} \Big)^2 + {\frac{\partial u_1}{\partial y_2}}^2 \right\} dxdy_1dy_2}\nonumber \\
&= & \int_{W} \left\{ \frac{\hat{f}u - u^2}{l(x)^2} + \frac{1}{l(x)^2}\Big( u \frac{\partial u}{\partial x} \frac{l'(x)}{l(x)} +  u \frac{\partial u_1}{\partial y_1} l'(x)\Big) \right\} dx dy_1 dy_2.
\end{eqnarray}
Then, combining (\ref{conver1}) and (\ref{conver2}) we have
\begin{eqnarray}\label{conver3}
\lefteqn{\int_{(0, 1) \times Y^*} \frac{1}{l(x)^2}\left\{{\mathcal{T_\eps}\Big(\frac{\partial u^\eps}{\partial x}\Big)}^2 + {\mathcal{T_\eps}\Big(\frac{\partial u^\eps}{\partial y}\Big)}^2\right\}
dx dy_1 dy_2} \nonumber \\
&\stackrel{\eps\to 0}\longrightarrow & \int_W \left\{  \Big(\frac{1}{l(x)}{\frac{\partial u}{\partial x}} + \frac{\partial u_1}{\partial y_1} \Big)^2 + {\frac{\partial u_1}{\partial y_2}}^2\right\} dxdy_1dy_2.
\end{eqnarray}
Consequently, due to the weak convergences (\ref{dev1}), (\ref{dev2}) and the convergence (\ref{conver3}) we can ensure by the Radon-Riez property the strong convergence \textit{ii)}.

 \item[iii)] This convergence follows from strong convergence ii) and using the property ii) and ii) of $\mathcal{U_\eps}$ in Proposition \ref{u}.
 \item[iv)] Due to the regularity of the functions $G(\cdot, \cdot)$ and $l(\cdot)$ we can ensure that the function $X$ belongs, at least, to $H^1\Big((0,1); C^1_{\#}\big({Y^*(x)}\big)\Big)$, to prove that $X$ has at least one derivative with respect to $x$ techniques of perturbation of the domains can be used,  see \cite[Proposition A.1]{ArrPer2011}, \cite[Proposition A.2]{MR} and \cite[Chapter 2]{He} where more general results can be found. Then, the function $X(x, x/\eps, y / \eps)$ is well-defined function in $H^1(R^\eps)$ and we can obtain some estimates in $R^\eps$. It is easy to see that
  \begin{eqnarray}\label{bound1} 
&& ||| X(x, x/\eps, y / \eps)||||^2_{L^2(R^\eps)} = \frac{1}{\eps}\int_{R^\eps} |X(x, x/\eps, y / \eps)|^2 \; dxdy  \nonumber\\
&& = \int_{(0,1)}\int_{(0, G(x, x/ \eps))} |X(x, x/\eps, z)|^2 \; dxdz \leq   C \int_{(0,1)} \sup_{(y_1, y_2) \in Y^*(x)} |X(x, y_1, y_2)|^2 \; dx\nonumber\\
&& = C || X||^2_{L^2\big((0,1); C^1_{\#}\big({Y^*(x)}\big) \big)}.
  \end{eqnarray}
 In order to simplify the notation we consider the following functions,
 $$X_{x}^{\eps}(x,y) = \frac{\partial X}{\partial x} \big(x, \frac{x}{\eps}, \frac{y}{ \eps}\big), \; X_{1}^{\eps}(x,y) = \frac{\partial X}{\partial y_1} \big(x, \frac{x}{\eps}, \frac{y}{ \eps}\big) \hbox{ and } X_{2}^{\eps}(x,y) = \frac{\partial X}{\partial y_2} \big(x, \frac{x}{\eps}, \frac{y}{ \eps}\big), \; \forall (x,y) \in R^\eps.$$ 
  Analogously to (\ref{bound1}), we can get
 \begin{eqnarray}\label{bound2}
 ||| X_{x}^{\eps}(x,y)||||^2_{L^2(R^\eps)} \leq C || \partial_{x}X||_{L^2\Big((0,1); C^1_{\#}\big({Y^*(x)}\big) \Big)}.
 \end{eqnarray}

  By the definition of the norm $|||\cdot|||_{H^1( R^\eps)}$ we have
 \begin{eqnarray}\label{corrector}
\lefteqn{ ||| u^\eps - u + \eps \frac{\partial u}{\partial x}X(x, x/\eps, y / \eps)|||^2_{H^1(R^\eps)} =  ||| u^\eps - u + \eps \frac{\partial u}{\partial x}X(x, x/\eps, y / \eps)|||^2_{L^2(R^\eps)}} \nonumber\\
 &+ & ||| \frac{\partial u^\eps}{\partial x} - \frac{\partial u}{\partial x} + \eps \frac{\partial u}{\partial x}X_{x}^{\eps}(x,y)  + \frac{\partial u}{\partial x}X_{1}^{\eps}(x,y)
 + \eps \frac{\partial^2 u}{\partial x^2}X(x, x/\eps, y / \eps)|||^2_{L^2(R^\eps)} \nonumber \\
& + & ||| \frac{\partial u^\eps}{\partial y} + \frac{\partial u}{\partial x}X_{2}^{\eps}(x,y)|||^2_{L^2(R^\eps)}.
  \end{eqnarray}
 
 Now, we calculate the limit for each term. For the first term we have the following inequality:
  
$$||| u^\eps - u + \eps \frac{\partial u}{\partial x}X(x, x/\eps, y / \eps)|||_{L^2(R^\eps)} \leqslant  ||| u^\eps - u|||_{L^2(R^\eps)} + |||\eps \frac{\partial u}{\partial x}X(x, x/\eps, y / \eps)|||_{L^2(R^\eps)}.$$ 
 
Therefore,  due to convergence i) and bound (\ref{bound1})  we get
\begin{equation}\label{term1}
||| u^\eps - u + \eps \frac{\partial u}{\partial x}X(x, x/\eps, y / \eps)|||_{L^2(R^\eps)} \to 0.
\end{equation}
 
 For the second term,  adding and substracting the appropiate functions and with the triangular inequality we obtain
 $$||| \frac{\partial u^\eps}{\partial x} - \frac{\partial u}{\partial x} + \eps \frac{\partial u}{\partial x}X_{x}^{\eps}(x,y)  + \frac{\partial u}{\partial x}X_{1}^{\eps}(x,y)
 + \eps \frac{\partial^2 u}{\partial x^2}X(x, x/\eps, y / \eps)|||_{L^2(R^\eps)}  \leq I_1 + I_2 + I_3,$$
 where
 \begin{eqnarray*}
 &&I_1 =  ||| \frac{\partial u^\eps}{\partial x} - \frac{\partial u}{\partial x} - l(x)\mathcal{U_\eps}(\frac{\partial u_1}{\partial y_1})|||_{L^2(R^\eps)},\\
 &&I_2 =  ||| l(x)\mathcal{U_\eps}(\frac{\partial u_1}{\partial y_1}) +\frac{\partial u}{\partial x}X_{1}^{\eps}(x,y)|||_{L^2(R^\eps)},\\
 &&I_3 = |||\eps \frac{\partial u}{\partial x}X_{x}^{\eps}(x,y) + \eps \frac{\partial^2 u}{\partial x^2}X(x, x/\eps, y / \eps)|||_{L^2(R^\eps)}.\\
 \end{eqnarray*}
 
 It follows from convergene iii)  that 
  $$I_1= ||| \frac{\partial u^\eps}{\partial x} - \frac{\partial u}{\partial x} - l(x)\mathcal{U_\eps}(\frac{\partial u_1}{\partial y_1})|||_{L^2(R^\eps)} \to 0, \; \hbox{as} \; \eps \to 0. $$
 Using Proposition \ref{convergence test} we have
$$ \mathcal{T_\eps}(X_{1}^{\eps}(x,y)) \longrightarrow \partial_{y_1}X \chi \quad \hbox{s}-L^2\big( (0,1)\times Y^*\big).$$
Consequently, from property v) in Proposition \ref{u} we obtain
$$ |||X_{1}^{\eps}(x,y)- \mathcal{U_\eps}(\partial{ y_1} X) |||_{L^2(R^\eps)} \to 0, \; \hbox{as} \; \eps \to 0.$$
  Therefore,  from (\ref{u_1}) we can conclude that
  $$I_2 =  ||| l(x)\mathcal{U_\eps}(\frac{\partial u_1}{\partial y_1}) +\frac{\partial u}{\partial x}X_{1}^{\eps}(x,y)|||_{L^2(R^\eps)}\to 0, \; \hbox{as} \; \eps \to 0. $$
  
  Moreover, by (\ref{bound1}) and (\ref{bound2}) we have 
  $$I_3 = |||\eps \frac{\partial u}{\partial x}X_{x}^{\eps}(x,y) + \eps \frac{\partial^2 u}{\partial x^2}X(x, x/\eps, y / \eps)|||_{L^2(R^\eps)} \to 0, \; \hbox{as} \; \eps \to 0.$$
  
  Then, we have proved that 
  \begin{equation}\label{term2}
  ||| \frac{\partial u^\eps}{\partial x} - \frac{\partial u}{\partial x} + \eps \frac{\partial u}{\partial x}X_{x}^{\eps}(x,y)  + \frac{\partial u}{\partial x}X_{1}^{\eps}(x,y)
 + \eps \frac{\partial^2 u}{\partial x^2}X(x, x/\eps, y / \eps)|||_{L^2(R^\eps)} \to 0.
 \end{equation}
 
 Finally, arguing as for the second term, tacking into account that $u$ does not depend on $y$, we obtain 
 \begin{equation}\label{term3}
 ||| \frac{\partial u^\eps}{\partial y} + \frac{\partial u}{\partial x}X_{2}^{\eps}(x,y)|||_{L^2(R^\eps)}\to 0.
 \end{equation}
 
 Therefore, in light of (\ref{corrector}), (\ref{term1}), (\ref{term2}) and (\ref{term3}) we prove iv).
  \end{enumerate}
 \end{proof}

\end{document}